\documentclass[11pt]{amsart}
\usepackage[margin=1.00in]{geometry}
\usepackage{amscd,amssymb, amsmath, wasysym, yfonts}
\usepackage{graphicx}
\usepackage[all]{xy}

\usepackage{url}
\usepackage{hyperref}

\theoremstyle{plain}
\newtheorem{theorem}{Theorem}[section]
\newtheorem{prop}[theorem]{Proposition}

\newtheorem{cor}[theorem]{Corollary}

\theoremstyle{definition}

\newtheorem{notation}[theorem]{Notation}
\newtheorem{rmk}[theorem]{Remark}
\newtheorem{set}[theorem]{Setting}

\newtheorem*{ex*}{Example}

\newcommand\sO{{\mathcal O}}
\newcommand\sK{{\mathcal K}}

\newcommand\sC{{\mathcal C}}
\newcommand\sP{{\mathcal P}}
\newcommand\sF{{\mathcal F}}
\newcommand\sS{{\mathcal S}}
\newcommand\sE{{\mathcal E}}
\newcommand\sG{{\mathcal G}}
\newcommand\sU{{\mathcal U}}
\newcommand\sA{{\mathcal A}}

\newcommand\sM{{\mathcal M}}
\newcommand\sB{{\mathcal B}}
\newcommand\sL{{\mathcal L}}

\newcommand\sY{{\mathcal Y}}

\newcommand\rp{{\mathbf{P}}}
\newcommand\rz{{\mathbf{Z}}}
\newcommand\rr{{\mathbf{R}}}

\newcommand\rc{{\mathbf{C}}}

\newcommand\rl{{\mathbf{L}}}

\newcommand\rd{{\mathbf{D}}}

\title[Derived equivalences of canonical covers]
{Derived equivalences of canonical covers of hyperelliptic and Enriques surfaces in positive characteristic}

\author{Katrina Honigs, Luigi Lombardi and Sofia Tirabassi}

\address{Department of Mathematics\\ The University of Utah, 155 S 1400 E, Salt Lake City, UT 84112-0090, USA}
 \email{\url{honigs@math.utah.edu}}
 
\address{Department of Mathematics\\ Stony Brook University, Stony Brook, NY, 11794-3651, USA}
 \email{\url{luigi.lombardi@stonybrook.edu}}

\address{Department of Mathematics\\ University of Bergen, Allegaten 52, Bergen, Norway}
\email{\url{sofia.tirabassi@uib.no}}

\begin{document}
\maketitle

\begin{abstract}
We prove that any Fourier--Mukai partner of an abelian surface over an algebraically closed field of positive characteristic is isomorphic to a moduli space of Gieseker-stable sheaves. 
We apply this fact to show that the Fourier--Mukai set of canonical covers of 
hyperelliptic and Enriques surfaces over an algebraically closed field of characteristic greater than three is trivial. 
These results extend to positive characteristic earlier results of 
Bridgeland--Maciocia and Sosna.
\end{abstract}

\section{Introduction}
The main motivation of this paper is the recent series of results in the study of equivalences of derived categories of sheaves of 
smooth projective varieties over fields other than the field of complex numbers. 
For instance, the first named author proves that the Zeta function of an abelian
variety, as well as of smooth varieties of dimension at most three, is unaltered under equivalences 
of derived categories \cite{Ho1,Ho3}. Moreover, Ward in his thesis \cite{Wa}
produces examples of genus one curves over $\mathbf{Q}$ admitting an arbitrary number of distinct Fourier--Mukai partners, revealing 
in this way consistent differences with the case of elliptic curves over $\rc$. 
Furthermore, Ward also studies arithmetic aspects of Calabi--Yau threefolds in positive characteristic.
Finally, we mention that Lieblich and Olsson in \cite{LO} extend to positive characteristic seminal works of Mukai and Orlov concerning 
 derived equivalences of $K3$ surfaces. In particular, they prove that any Fourier--Mukai partner of a $K3$ surface $X$ 
 over an algebraically closed field of characteristic $p\neq 2$ is a moduli space of Gieseker-stable sheaves on $X$, and in addition 
 $X$ admits only a finite number of Fourier--Mukai partners.
While Orlov's proof relies on Hodge theory, Lieblich--Olsson's proof relies on deformation theory of perfect complexes 
and on the theory of liftings to 
the Witt ring, that allows them to lift the whole problem in characteristic zero where Orlov's results can be applied. 

Inspired by \cite{LO}, we seek to study the set of Fourier--Mukai partners of surfaces 
defined in positive characteristic, such as its finiteness and its members. 
In this paper we focus on a special class of abelian and $K3$ surfaces that arise as canonical covers of hyperelliptic and Enriques surfaces. 
Our first main result is that derived equivalent canonical covers of hyperelliptic surfaces are isomorphic.
This extends to positive characteristic the work of Sosna \cite[Theorem 1.1]{So}.
\begin{theorem}\label{main1}
 Let $S$ be an hyperelliptic surface over an algebraically closed field of characteristic $p>3$ and let $A$ be its canonical cover.
 Then any smooth projective surface that is derived equivalent to $A$ is isomorphic to either $A$ or its dual $\widehat{A}$. 
\end{theorem}
We refer to Theorem \ref{MTg} for a slightly stronger result.
The work of Orlov in the study of derived equivalences of abelian varieties shows that any two abelian varieties $A$ and $B$ 
are derived equivalent if and only if there exists a symplectic isomorphism between the products 
$A\times \widehat{A}$ and $B\times \widehat{B}$ (\cite{Or}). In particular, 
derived equivalent abelian varieties are isogenous and admit only a finite number of partners, 
even in positive characteristic (\cite[Corollary 4.1.3]{Ho2}). 
However this concrete picture involving symplectic isomorphisms did not lead us too far towards the solution of 
our problem; in fact this 
was already observed by Sosna who grounds his proofs on Hodge theory and lattice theory. On the other hand our main ingredient is the 
characterization of  Fourier--Mukai partners of abelian surfaces in positive characteristic as moduli spaces of Gieseker-stable sheaves. 
The following result extends to positive characteristic the result \cite[Theorem 5.1]{BM} of Bridgeland--Maciocia. 
In the following we denote by $\rd(X)$ the bounded derived category of coherent sheaves on a smooth projective variety $X$.
\begin{theorem}\label{main2}
Let $A$ be an abelian surface over an algebraically closed field $k$ of positive characteristic and 
let $Y$ be a smooth projective variety over $k$. Suppose furthermore that there is an equivalence of triangulated 
categories $\Phi:\rd(A)\rightarrow \rd(Y)$.
 Then $Y$ is an abelian surface and $A$ is isomorphic 
 to a moduli space of Gieseker-stable sheaves on either $Y$ or its dual $\widehat{Y}$.
\end{theorem}
The proof of the previous theorem is centered on the notion of \emph{filtered equivalences}. 
As equivalences of derived categories attached to smooth projective varieties are of Fourier--Mukai type, a derived equivalence $\Phi:\rd(X)\rightarrow \rd(Y)$ of surfaces 
induces a homomorphism $\Phi^{\rm CH}:{\rm CH}^*(X)\rightarrow {\rm CH}^*(Y)$ between the numerical 
Chow rings which in general does not respect the grading. 
Then one says that $\Phi$ is filtered if $\Phi^{\rm CH}(0,0,1)=(0,0,1)$. 
In Proposition \ref{filteredequivalence} we show that in the case of abelian surfaces, 
a filtered equivalence induces an isomorphism between the surfaces themselves. This mainly follows from the fact 
that derived equivalences of abelian 
varieties enjoy the property that skyscraper sheaves are sent to sheaves up to shift. 
Therefore in order to complete the proof of Theorem \ref{main2} we construct an equivalence 
$$\Xi:\rd(A)\stackrel{\Phi}\longrightarrow \rd(Y)\stackrel{\Psi}{\longrightarrow} \rd(Y)$$ 
as the composition of $\Phi$ with an autoequivalence $\Psi$ of $\rd(Y)$
such that $\Xi^{\rm CH}(0,0,1)$ is a vector $v=(r,l,\chi)$ for which 
the corresponding moduli space $\sM_Y(v)$ of Gieseker-stable sheaves is a smooth surface that admits a universal family $\sU$. 
This completes the proof as the composition of $\Xi$ with the Fourier--Mukai functor associated to $\sU$ is a filtered equivalence.
However we have been able to carry out this plan only if the rank component of $\Phi^{\rm CH}(0,0,1)$ is non-zero, while in the remaining case 
we can perform the same strategy at the price of involving the Mukai's equivalence $\sS_Y:\rd(Y)\rightarrow \rd(\widehat{Y})$ 
in the construction of $\Psi$. 
This explains why the conclusions of Theorem \ref{main2} and \cite[Theorem 5.1]{BM} are not completely symmetric. 
Finally, we point out that filtered equivalences have been introduced in \cite{LO}, 
together with other stronger versions in \cite{LO2}, 
to establish a derived version of Torelli theorem for $K3$ surfaces in characteristic $p\neq 2$.

Now we go back to the proof of Theorem \ref{main1}. The starting point is Bagnera--De Franchis' list that 
realizes a hyperelliptic surface $S$ as a quotient of a product of 
two elliptic curves $E\times F$ by a finite group $G$. In particular one deduces that 
the dual of the canonical cover $\widehat{A}$ of $S$ is either a product of two elliptic curves, or 
an \'{e}tale cover of degree two of two elliptic curves, or an \'{e}tale cover of degree three of 
two elliptic curves of which one admits an
automorphism of groups of order three. The strategy now is to lift to the Witt ring both the elliptic curves $E$ and $F$, the \'{e}tale 
cover $\widehat{A}\rightarrow E\times F$, and an arbitrary Fourier--Mukai partner $B$ of $\widehat{A}$ so that, upon restricting to the geometric generic fibers, one can 
involve Sosna's result in characteristic zero in order to get an isomorphism between the closed fibers of the lifts of $B$ and $\widehat{A}$. 
We remark that the main difficulty in doing so is that lifts should be carefully chosen. In fact one needs to involve arguments similar to the ones
of the proof of Theorem \ref{main2} in order to build a lift of $B$ as a relative moduli space, 
and in addition, in the case of covers of degree three, one needs to lift an elliptic curve together with 
an automorphism and a principal polarization. 
This is possible because elliptic curves, both 
ordinary and supersingular, have good lifting properties. 
In particular we use the theory of canonical covers in the case of ordinary elliptic curves,
and a result of Deuring (\emph{cf}. \cite{Oo}) 
that allows us to lift to a ramified extension of the Witt ring a supersingular elliptic curve together with an automorphism.
Furthermore, another peculiar fact shared by supersingular elliptic curves that we need of, is that the  
graph of isogenies of supersingular curves is connected (\cite[Corollary 78]{Kohe}).

In the last section we observe that one can push a little further the techniques of \cite{LO} in order to prove that $K3$ surfaces that are 
canonical covers of Enriques surfaces in characteristic $p>3$ do not admit any non-trivial Fourier--Mukai partner. This in particular extends to 
positive characteristic the second part of the result of Sosna \cite[Theorem 1.1]{So}.
\begin{theorem}\label{main3}
  Let $S$ be an Enriques surface over an algebraically closed field of characteristic $p>3$ and let $X$ be its canonical cover.
  Then any smooth projective surface that is derived equivalent to $X$ is isomorphic to $X$.
\end{theorem}

\subsection*{Notation} Unless otherwise specified we work over an algebraically closed field $k$ of positive characteristic $p$. Further notation is introduced in Notation \ref{notat}.

\subsection*{Acknowledgements}
This project started when ST visited LL in the Summer of 2014 at the University of Bonn. LL and ST heartily thank Daniel Huybrechts for 
guiding them during the early stages of the project and for his insights. Moreover LL and ST thank
the Institute of Mathematics of the University of Bonn for the optimal working conditions offered. 
The authors all thank Martin Olsson for useful advices and Pawel Sosna for carefully reading a first draft of this paper. Moreover
KH thanks Aaron Bertram and Eric Canton for helpful conversation;
LL thanks Alberto Bellardini, Christian Liedtke, Christian Schnell and Mattia Talpo for  
comments and correspondence; 
ST thanks A. Bertram, C. Hacon, M. Lieblich, B. Moonen, K. Schwede, M. Talpo and B.
Viray for very interesting mathematical conversation. Finally the authors thank their previous institutions that respectively were:
University of California at Berkeley, University of Bonn, and the University of Utah. 
LL was partly supported by the SFB/TR45 ``Periods, moduli spaces, and arithmetic of algebraic varieties'' of the DFG (German Research Foundation).

%
%

\section{Background material}

\subsection{Fourier--Mukai transforms and Chow rings}
  
 Let $k$ be an algebraically closed field of positive characteristic $p$. 
 The bounded derived category of sheaves of a smooth projective variety $X$ is defined as 
 $\rd(X):=D^b\big(\mathcal{C}oh(X) \big)$. The category $\rd(X)$ is $k$-linear and triangulated. 
 If $Y$ is another smooth projective variety, an object $\sE$ in $\rd(X\times Y)$ 
  defines a \emph{Fourier--Mukai functor} via the assignment:
$$\Phi_{\sE} : \rd(X) \to \rd(Y), \quad \sF\mapsto \rr p_{2*}\big(p_1^*\sF\stackrel{\rl}{\otimes} \sE\big)$$ 
\noindent where $p_1$ and $p_2$ denote the projections from $X\times Y$ onto the first and second factor respectively. 
  An important theorem of Orlov tells us that any equivalence $F:\rd(X)\rightarrow \rd(Y)$ is of Fourier--Mukai type, \emph{i.e.}
  there exists a unique up to isomorphism \emph{kernel} $\sE$ in $\rd(X\times Y)$ such that $F\simeq \Phi_{\sE}$. Finally we recall that the 
  composition of Fourier--Mukai transforms is again of Fourier--Mukai type.

  Consider now an abelian surface $A$ over $k$. We denote by
${\rm CH}^*(A)_{{\rm num}}=\oplus_i {\rm CH}^i(A)_{{\rm num}}$ 
the graded ring of algebraic cycles modulo \emph{numerical} equivalence so that 
$${\rm CH}^0(A)_{\rm num}\simeq \rz, \quad {\rm CH}^1(A)_{\rm num}\simeq {\rm NS}(A)\quad \mbox{and}\quad 
{\rm CH}^2(A)_{\rm num}\simeq \rz,$$ where ${\rm NS}(A)$ denotes the N\'{e}ron--Severi group of $A$ 
up to torsion. Moreover we set ${\rm CH}^*(A)_{{\rm num},\mathbf{Q}}:={\rm CH}^*(A)_{{\rm num}}\otimes \mathbf{Q}$.
For an object $\sF$ in $\rd(A)$ we denote by $v(\sF)\in {\rm CH}^*(A)_{{\rm num},\mathbf{Q}}$ its \emph{Mukai vector} 
(see \cite[\S 5.2]{Huy} and \cite[p. 3]{BBHP}). 
 Hence the Mukai vector of a locally free sheaf $E$ on $A$ is 
 $v(E)=\big({\rm rk}(E), \, c_1(E),\, \chi(E)\big)$ and the map $$v:\rd(A)\rightarrow  {\rm CH}^*(A)_{{\rm num},\mathbf{Q}}$$ factors 
 through the Grothendieck group $K(A)$ of locally free sheaves via the Chern character
$${\rm ch}: K(A) \; \rightarrow {\rm CH}^*(A)_{{\rm num},\mathbf{Q}}.$$
 Finally, we denote the \emph{Mukai pairing} on
${\rm CH}^*(A)_{{\rm num},\mathbf{Q}}$ by
$$\langle (r,l,\chi) \, , \, (r',l',\chi')\rangle_A \; := \; l\cdot l' \,- \, r \,\chi' \, - \, \chi \, r',$$ so that
by the Grothendieck--Riemann--Roch Theorem there are equalities
\begin{equation}\label{isometry1}
 \langle v(\sF), \, v(\sG)\rangle_A \; = \; -\,\chi(\sF,\sG)\quad \mbox{ for any objects } \sF,\sG \mbox{ in } \rd(A) 
\end{equation}
(as usual $\chi(\sF,\sG)=\sum_{i} (-1)^i {\rm dim}\, {\rm Hom}_{\rd(A)}^i(\sF,\sG)$). 

Given another abelian surface $B$, a Fourier--Mukai functor $\Phi_{\sE}:\rd(A)\rightarrow \rd(B)$ induces a homomorphism of rings 
$\Phi_{\sE}^{{\rm CH}}:{\rm CH}^*(A)_{{\rm num},\mathbf{Q}} \rightarrow {\rm CH}^*(B)_{{\rm num},\mathbf{Q}}$ 
through the formula
$$\Phi_{\sE}^{{\rm CH}}(-) \; := \; {\rm pr}_{2*}\big({\rm pr}_1^*(-) \, \cdot \, v(\sE)\big)$$ 
where ${\rm pr}_1$ and ${\rm pr}_2$ denote the projections from the product $A\times B$ onto the first and second factor respectively. 
We point out that in general $\Phi_{\sE}^{\rm CH}$ does not respect the grading.
As showed in \cite{Huy} and \cite[\S 3]{Ho2}, it is possible to prove that 
$\big(\Phi_{\sE}\circ \Phi_{\sE'}\big)^{{\rm CH}} \simeq \Phi_{\sE}^{\rm CH} \circ \Phi_{\sE'}^{\rm CH}$ and that $\Phi_{\sE}^{\rm CH}$ is 
invertible if $\Phi_{\sE}$ is an equivalence. Finally we note that if $\Phi_{\sE}$ is an equivalence, then it follows from \eqref{isometry1}, and by the fact that $v\circ \Phi_{\sE}=\Phi_{\sE}^{\rm CH}\circ v$, 
that for any $\sF$ and $\sG$ in $\rd(A)$ 
\begin{equation}\label{isometry}
\big \langle \Phi_{\sE}^{\rm CH} \big(v(\sF) \big), \, \Phi_{\sE}^{\rm CH}\big(v(\sG) \big) \big \rangle_B \; = \; 
\big\langle v(\sF) \, ,\, v(\sG) \big\rangle_A.  
\end{equation}

\noindent We conclude this subsection by pointing out the following peculiar fact true for abelian surfaces. 
Its proof is identical to the proof of \cite[Corollary 9.43]{Huy} with the opportune modifications, and moreover it holds in any dimension. We 
will tacitly use the following result throughout the rest of the paper.
\begin{prop}\label{chinteger}
If $\Phi:\rd(A)\rightarrow \rd(B)$ is an equivalence of derived categories of abelian surfaces, 
then 
$$\Phi^{\rm CH}\big({\rm CH}^*(A)_{\rm num}\big) \; = \; {\rm CH}^*(B)_{\rm num}.$$ 
\end{prop}
 
\begin{notation}\label{notat}
Given an abelian surface $A$ over $k$ 
we denote by ${\rm CH}^*(A)=\oplus_i {\rm CH}^i(A)$ the graded ring ${\rm CH}^*(A)_{{\rm num},\mathbf{Q}}$. 
\end{notation}

\subsection{Some examples of (auto)equivalences} \label{autoeq}
We denote by $A$ an abelian surface and by $\widehat{A}$ its dual variety.
Moreover let $\sP$ be the normalized Poincar\'{e} line bundle on $A\times \widehat{A}$ so that 
$\sS_A:=\Phi_{\sP}:\rd(A)\rightarrow \rd(\widehat{A})$ is an equivalence of triangulated categories \cite{Muk1}. 
The action of $\sS_A^{\rm CH}$
swaps the first and third entry of a vector, \emph{e.g.}:
$$\sS_{A}^{\rm CH}\big({\rm CH}^0(A)\big)={\rm CH}^2(\widehat{A}),\quad  \sS_{A}^{\rm CH}\big({\rm CH}^2(A)\big)={\rm CH}^0(\widehat{A}), 
\quad \mbox{and}\quad \sS_{A}^{\rm CH}\big({\rm CH}^1(A)\big)={\rm CH}^1(\widehat{A}).$$
Let now $H$ be a line bundle on $A$ and $h$ be its class in ${\rm CH}^1(A)$. The autoequivalence 
$T_A(H^{\otimes n}):\rd(A)\rightarrow \rd(A)$ ($n\in \rz$) defined by 
$\sF\mapsto \sF\otimes H^{\otimes n}$ acts on the numerical 
Chow rings as: 
\begin{equation}\label{actionH}
 T_A(H^{\otimes n})^{{\rm CH}}(r,l,\chi) \, = \, \Big(r, \, l\, +\, r\,n\,h,
 \, \chi\, +\, n\,l\cdot h\,+ \,r\,n^2 \, \frac{h^2}{2}\Big).
 \end{equation}
 Finally, the shift functor $[1]:\rd(A)\rightarrow \rd(A)$ acts on ${\rm CH}^*(A)$ by $-1$.

\subsection{Isogenies and exponents}
If $A$ is an abelian variety over $k$,
we denote by $n_A:A\rightarrow A$ the multiplication-by-$n$-map on $A$ and by
$A[n]$ the kernel of $n_A$.
We say that an elliptic curve $E$ over $k$ is \emph{ordinary} if 
$E[p](k)=\rz / p\rz$  and \emph{supersingular}
if $E[p](k)=0$ (\emph{cf}. \cite{Oo2}). Therefore $E$ is supersingular 
if and only if $p_E$ is inseparable and the $j$-invariant
is defined over $\mathbf{F}_{p^2}$, the finite field with $p^2$ elements (\emph{cf}. \cite[Theorem V.3.1]{Si}).
The \emph{exponent} ${\rm exp} \, \varphi$ of a separable 
isogeny $\varphi:A\rightarrow B$ of abelian varieties is the smallest positive integer
that annihilates its kernel. 
Finally we recall that if $\varphi:A\rightarrow B$ is a separable isogeny of exponent $e$, 
then there exists a separable isogeny $\psi:B\rightarrow A$ of exponent
$e$ such that $\psi\circ \varphi=e_A$ and $\varphi\circ \psi=e_B$ (\emph{cf}. \cite[Proposition 1.2.6]{BL}).

\begin{prop}\label{Neron}
 Let $\nu:A\rightarrow B$ be a separable isogeny of exponent $e$ and denote by $\nu^*:{\rm CH}^1(B)\rightarrow {\rm CH}^1(A)$ 
 the pull-back homomorphism. Then there is an inclusion of groups
 $e^2{\rm CH}^1(A)\subset {\rm Im}(\nu^*)$.
\end{prop}

\begin{proof}
  Let $\mu:B\rightarrow A$ be the isogeny such that $\mu \circ \nu = e_A$ and note that
  ${\rm Im}\big(e_A^*:{\rm CH}^1(A)\rightarrow {\rm CH}^1(A) \big)\subset {\rm Im}(\nu^*)$.
  We conclude by using \cite[Corollary 7.25]{VdGM} which shows that $e_A^*$ is the multiplication-by-$e^2$-map.
  \end{proof}

\begin{prop}\label{koh}
 If $E$ and $F$ are supersingular elliptic curves and $l\neq p={\rm char}(k)$ is a prime,
 then there exist an integer $r\gg 0$ and a separable isogeny $\xi:F\rightarrow E$ of degree $l^r$.
\end{prop}

\begin{proof}
 Since $E$ and $F$ are supersingular, their $j$-invariants are defined over $\mathbf{F}_{p^2}$. 
 Moreover by \cite[Corollary 78]{Kohe} there exists an isogeny $\xi':F(\mathbf{F}_{p^2})\rightarrow E(\mathbf{F}_{p^2})$ of degree $l^r$ for some 
 positive integer $r\gg 0$. 
 Therefore we obtain our desired isogeny from $\xi'$ by extension of scalars. 
 Finally we observe that $\xi$ is separable 
 as the degree of every non-separable isogeny is divisible by ${\rm char}(k)$ (\cite[Corollary 2.12]{Si}).
 \end{proof}

\subsection{Line bundles on a product of two elliptic curves}\label{mumlin}
Let $(E,O_E)$ be an elliptic curve over $k$. We denote the \emph{Mumford bundle} on $E\times E$ by 
$$\sM_E\; \:= \; \sO_{E\times E}(\Delta_E)\, \otimes \, {\rm pr}_1^*\sO_E(-O_E)\, \otimes \, {\rm pr}_2^*\sO_E(-O_E)$$
where $\Delta_E \subset E\times E$ is the diagonal divisor and ${\rm pr}_1$, ${\rm pr}_2$ are the projections of $E\times E$ onto 
the first and second factor respectively. Given another elliptic curve $(F,O_F)$, line bundles $L_E$ and $L_F$ on $E$ and $F$ respectively,
and a morphism $\varphi:F\rightarrow E$, we define a line bundle on the product $E\times F$
\begin{equation}\label{Picard}
  L(\varphi,L_E,L_F) \, := \, (1_E\times \varphi)^*\sM_E \otimes {\rm pr}_E^*L_E\otimes {\rm pr}_F^*L_F
 \end{equation}
 where ${\rm pr}_E$ and ${\rm pr}_F$ are the projections onto $E$ and $F$ respectively.

\begin{prop}\label{Picardprod2}
  If $\varphi:F\rightarrow E$ and $\psi:F\rightarrow E$ are isogenies, 
 then $$\big(1_E\times (\varphi+\psi)\big)^*\sM_E\, \simeq \, (1_E\times \varphi)^*\sM_E\otimes (1_E\times \psi)^*\sM_E.$$
Therefore for any choice of line bundles $M_E$ and $N_E$ on $E$, and line bundles $M_F$ and $N_F$ on $F$, there are isomorphisms
  $$L(\varphi+\psi,M_E\otimes N_E,M_F\otimes N_F)\, \simeq \, L(\varphi,M_E,M_F) \otimes L(\psi,N_E,N_F).$$
  
  \end{prop}

  \begin{proof}
 The proof is a simple application of the see-saw principle.
\end{proof}

\noindent If $L_E,L_E'$ and $L_F,L_F'$ are line bundles on $E$ and $F$ respectively such that $a_E:=\deg\,L_E=\deg\, L_E'$ and $a_F=\deg\,L_F=\deg\,L_F'$, then 
in ${\rm CH}^1(E\times F)$ the classes of $L(\varphi,L_E,L_F)$ and $L(\varphi,L_E',L_F')$ coincide. We denote then by  
 $l(\varphi,d_E,d_F)$ the numerical class of $L(\varphi,L_E,L_F)$ (or of $L(\varphi,L_E',L_F')$).

\begin{cor}\label{corPicardprod}
 With notation as in Proposition \ref{Picardprod2}, in ${\rm CH}^1(E\times F)$ there are equalities of classes 
 $$l(\varphi+\psi,d_E+d'_E,d_F+d'_F)\, = \, l(\varphi,d_E,d_F) \, + \, l(\psi,d'_E,d'_F)$$
 where $d_E$ and $d_F$ are the degrees of $M_E$ and $M_F$ respectively, and $d'_E$ and $d'_F$ are the degrees of $N_E$ and $N_F$ respectively.
\end{cor}

\noindent Finally we show that any line bundle $L\in {\rm Pic}(E\times F)$ can be realized as a line bundle of the form \eqref{Picard}.

\begin{prop}\label{Picardprod}
  For any line bundle $L\in {\rm Pic}(E\times F)$ there exists a morphism 
 $\varphi:F\rightarrow E$ and line bundles $L_E\in {\rm Pic}(E)$ and $L_F\in {\rm Pic}(F)$ such that $L\simeq  L(\varphi,L_E,L_F)$. 
 \end{prop}

\begin{proof}
 Denote by $L_E$ the restriction of $L$ to $E\times \{O_F\}$, and similarly let $L_F$ be the restriction of $L$ to 
 $\{O_E\}\times F$. Set now 
 $L' \;:=\: L\, \otimes {\rm pr}_E^*L_E^{-1}\, \otimes \, {\rm pr}_F^* L_F^{-1}.$ We note that the restriction of $L'$ to 
 $\{O_E\}\times F$ is trivial, while the restrictions $L'|_{E\times \{y\}}$ lie in ${\rm Pic}^0(E)$ for all $y\in F$. Thus by the 
 universal property of the dual variety (\cite[Theorem on p. 117]{Mum}), 
 there exists a unique morphism $\widetilde \varphi:F\rightarrow \widehat{E}$ such that 
 $$L' \; = \; (1_E\times \widetilde\varphi)^*\sP_E$$
 where $\sP_E$ is the normalized Poincar\'{e} line bundle on $E\times \widehat{E}$ (namely the restrictions of $\sP_{E}$ to 
 $\{O_E\}\times \widehat{E}$ and $E\times \{O_{\widehat E}\}$ are trivial).
 Moreover $\sM_E\simeq (1_E\times \eta)^{*}\sP_E$ via the isomorphism $\eta(x)=\sO_E(x-O_E)$. Hence 
 $L'\simeq \big(1_E\times (\eta^{-1}\widetilde{\varphi}) \big)^*\sM_E$ and the conclusion follows by setting $\varphi=\eta^{-1}\widetilde{\varphi}$.
\end{proof}

\subsection{Lifting results}
Let $k$ be a perfect field of positive characteristic $p$ and let $W=W(k)$ be the ring of Witt vectors with quotient field $K$. 
We recall that $W$ is a complete discrete valuation ring such that $K$ is of 
characteristic zero (see for instance \cite[\S 11.1]{Li}). With $W$ we will also denote a finite ramified extension of the ring of Witt vectors $W(k)$. 
 If $X$ is a smooth projective scheme over $k$, we say that $\psi:\mathcal{X}\rightarrow W$ is a projective lift of $X$ 
if $\mathcal{X}$ is a projective scheme, the morphism $\psi$ is flat, and the closed fiber $\mathcal{X}_k$ is isomorphic to $X$. 
Grothendieck's existence theorem establishes that smooth curves always lift, as well as the line bundles on them. 
Moreover ordinary abelian varieties
admit a \emph{canonical lift} over $W$ characterized 
by the fact that the absolute Frobenious lifts sideways with the abelian variety (we recall that an 
abelian variety $A$ is ordinary if $A[p](k)\simeq (\rz / p\rz)^{\dim A}$). 
We refer to \cite[Appendix, Theorem 1]{MSN} for the proof of the following result.
\begin{theorem}\label{canonicallift}
Let $A$ be an ordinary abelian variety over a perfect field $k$ of positive characteristic $p$. Then there exists a projective lift
$\sA\rightarrow W$ of $A$ together with a morphism $F_{\sA}:\sA\rightarrow \sA$ compatible with the Frobenious of $W$ such that 
$F_{\sA|A}$ is the absolute Frobenious $F_A$ of $A$. The pair $(\sA,F_{\sA})$ is called \emph{canonical lift} and 
is unique up to a unique isomorphism inducing 
the identity on $A$. Moreover, the restriction morphism ${\rm Pic}(\sA)\rightarrow {\rm Pic}(A)$ is surjective
and $${\rm Pic}(\sA)_{F_{\sA}}:=\{\sL\in {\rm Pic}(\sA) \, | \, F_{\sA}^*\sL \simeq \sL^{\otimes p}   \}
\simeq {\rm Pic}(A).$$
Finally, if $\varphi:A\rightarrow B$ is a morphism between ordinary abelian varieties, then there exists a unique morphism 
$\widetilde{\varphi}:\sA\rightarrow \sB$ of canonical liftings such that $F_{\sB}\circ \, \widetilde{\varphi} = \, \widetilde{\varphi}\circ F_{\sA}$
and $\widetilde{\varphi}_{|A} = \varphi$.
\end{theorem}
\noindent Another result we will need in the sequel is the existence of liftings of \'{e}tale covers. A reference for the following theorem is 
\cite[\S IX, 1.10]{SGA}.
\begin{theorem}\label{etalecover}
Let $S$ be the spectrum of a complete local Noetherian ring, and let $X\rightarrow S$ be a proper $S$-scheme. 
Moreover denote by $X_0$ the closed fiber over the unique closed point of $S$. Then the assignment $X'\mapsto X'\times_{X}X_0$ 
yields an equivalence between the category of \'{e}tale coverings of $X$ and the category of \'{e}tale coverings of $X_0$.
\end{theorem}

\subsection{Moduli spaces}
Let $A$ be an abelian surface over an algebraically closed field $k$ and 
let $h\in {\rm NS}(A)$ be the class of an ample line bundle. Given a vector $v=(r,l,\chi) \in {\rm CH}^*(A)$ with integral coefficients, 
we consider the moduli space $\sM_h(v)$ of Gieseker-semistable sheaves with Mukai vector $v$, where stability is computed with respect to $h$. 

\begin{theorem}\label{modulispace}
If $r>0$ and $\chi$ are coprime integers, then every Gieseker-semistable sheaf on $A$ with Mukai vector $v$ 
is Gieseker-stable. Moreover, if in addition $\langle v,v \rangle_A=0$, then 
$\sM_h(v)$ is a smooth equidimensional projective variety of dimension two which admits a universal family $\sU$ on $\sM_h(v)\times A$.
Furthermore, every irreducible component $M$ of $\sM_h(v)$ has trivial canonical bundle and the Fourier--Mukai 
functor $\Phi_{\sU}:\rd(M)\rightarrow \rd(A)$ induces an equivalence of derived categories.
\end{theorem}

\begin{proof}
 The first assertion follows by \cite[Remark 4.6.8]{HL} (\emph{cf}. also \cite[Remark 6.1.9]{HL}), while the second follows
 by \cite[Corollary 0.2]{Muk3}.
  The existence of a universal sheaf follows by \cite[Theorem A.6]{Muk5}. 
 Now we show that the functor $\Phi_{\sU}:\rd(M)\rightarrow \rd(A)$ 
 induces an equivalence of derived categories for any irreducible component $M$ of $\sM_h(v)$. 
 By \cite[Proposition 3.12]{Muk5} the sheaf $\sU$ is strongly simple, and hence it yields a fully faithful Fourier--Mukai functor 
 $\Phi_{\sU}$ (\emph{cf}. \cite[Theorem 1.33]{BBHP}). However as both $A$ and $M$ are two-dimensional smooth varieties 
with trivial canonical bundles, we conclude by \cite[Corollary 7.8]{Huy} that 
 $\Phi_{\sU}$ is an equivalence.
 \end{proof}
 
 \subsection{Relative Moduli Spaces}\label{modulispacech} We also need to consider relative moduli spaces of 
 Gieseker-semistable sheaves on a projective lift $f:\sA\rightarrow W$ of $A$ over the ring of 
 Witt vectors. Let $h$ as before be the class of an ample line bundle, 
 and let $\widetilde{h}$ be a lifting of $h$ to $\sA$. Let $v=(r,l,\chi) \in {\rm CH}^*(A)$ 
 be a vector with integral coefficients such that $l$ is the class of a line bundle $L$ that lifts to 
 a line bundle $\widetilde{L}$ on $\sA$. Moreover set $\widetilde{v}=(r,\widetilde{l},\chi)$ where $\widetilde{l}$ is the class of $\widetilde{L}$.
 By \cite[Theorem 0.7]{Ma} (or \cite[Theorem 0.2]{La}) there exists a projective 
 scheme $\sM_{\sA/W}(\widetilde{v})\rightarrow W$ of finite type that is a coarse moduli space for 
 the functor of families of pure Gieseker-semistable sheaves
 with Mukai vector $v$ on the geometric fibers of $f$ (where stability is computed with respect to $\widetilde{h}$). 
 Moreover, there exists an open subscheme $\sM^s_{\sA/W}(\widetilde{v})\subset \sM_{\sA/W}(\widetilde{v})$ 
 that is a coarse moduli space for the subfunctor of families of pure Gieseker-stable sheaves. 
 Thus, if $(r,\chi)=1$ (\emph{i.e.} every Gieseker-semistable sheaf on any geometric fiber of $f$ is Gieseker-stable), then 
 $\sM^s_{\sA / W}(\widetilde{v})=\sM_{\sA/W}(\widetilde{v})$. Moreover, if 
 we denote by $A_k$ the closed fiber of $f$ and by $A_{\eta}$ the geometric generic fiber, then there are isomorphisms
 $$\sM_{\sA/W}(\widetilde{v})_{|A_k} \simeq \sM_h(v)\quad \mbox{ and }\quad \sM_{\sA/W}(\widetilde{v})_{|A_{\eta}} \simeq \sM_{h_{\eta}}(v_{\eta})$$
 where $\sM_{h_{\eta}}(v_{\eta})$ is the moduli space of pure Gieseker-stable sheaves with vector $v_{\eta}=(r,\tilde{l}_{|A_{\eta}}, \chi)$ 
 on $A_{\eta}$ 
 and $h_{\eta}$ is the restriction of $\widetilde{h}$ to $A_{\eta}$. Finally, if $\langle v,v\rangle_A=0$, 
 then $\sM_{h_{\eta}}(v_{\eta})$ is non-empty as by \cite{Lan}
 we have $\langle v_{\eta},v_{\eta}\rangle_{A_{\eta}}=\langle v,v\rangle_A=0$. 
 Thus the morphism $\widetilde{f}:\sM_{\sA/W}(\widetilde{v})\rightarrow W$ is flat and hence
 a projective lift of $M_{h}(v)$ over $W$.

\section{Filtered derived equivalences}
 
An equivalence $\Phi:\rd(A)\rightarrow \rd(B)$ of derived categories of abelian surfaces 
is \emph{filtered} if 
\begin{equation*}\label{filtereddef}
\Phi^{\rm CH} (0, \, 0, \, 1 \big)\; = \; (0, \, 0, \,1).
\end{equation*}
In \cite[Theorem 6.1]{LO} the authors prove that a filtered equivalence of $K3$ surfaces 
induces an isomorphism between them.
The proof of this statement is quite involved and uses deformation theory in order to 
lift the derived equivalence between $K3$ surfaces in positive characteristic, to an equivalence of $K3$ surfaces in characteristic zero. 
Here we notice that a filtered equivalence of abelian surfaces still induces an isomorphism. As the kernel of an equivalence of abelian 
varieties is a sheaf up to shift, its proof turns out to be rather simple.

\begin{prop}\label{filteredequivalence}
Let $\Phi:\rd(A)\rightarrow \rd(B)$ 
be a filtered equivalence of derived categories of two abelian surfaces.
Then there exists an isomorphism $f:A\rightarrow B$ and a line bundle $L$ such that there is an isomorphism
$\Phi\simeq (-\otimes L)\circ f_*$ up to shift. In particular $A$ and $B$ are isomorphic.
\end{prop}

\begin{proof}
Equivalences of derived categories of abelian varieties 
send (up to shift) structure sheaves of points $\sO_x$ to sheaves. 
This is proved in \cite[Lemma 10.2.6]{Br} in characteristic zero, but its proof extends to positive characteristic
without any change. Hence we can suppose that $\Phi(\sO_x)$ is a sheaf with Mukai vector $(0,0,1)$, so it is itself a skyscraper sheaf.
Since the argument holds for all points $x$ in $A$, the proposition follows by \cite[Corollary 5.23]{Huy}. 
\end{proof} 
 \noindent We now prove Theorem \ref{main2} of the Introduction which relies on the following technical proposition.

\begin{prop}\label{RelPrime3}
Let $\Phi:\rd(A)\rightarrow \rd(B)$ be an equivalence of derived categories of two abelian surfaces.  
Then there exists an equivalence $\Psi:\rd(A)\rightarrow \rd(C)$ where $C\in \{B,\widehat{B}\}$ such that the vector 
$$ \Psi^{\rm CH}(0, \,0, \,1)\; := \; (r, \, l, \, \chi) $$
satisfies the following conditions:
\begin{enumerate}
\item[(i).] $r$ is positive;\\
\item[(ii).] $l$ is the class of an ample line bundle on $C$;\\
\item[(iii).] $r$ is coprime with $\chi$.
\end{enumerate}
\end{prop}

\begin{proof} 
 We first prove that there exists an equivalence $\Psi_1:\rd(A)\rightarrow \rd(C)$ with $C\in \{B,\widehat{B}\}$ such 
 that the first entry of the vector
  $v_1:=\Psi_1^{{\rm CH}}(0,0,1)=(r_1, \, l_1, \, \chi_1)$ is positive. 
  Set $$v_0 \; := \; \Phi^{{\rm CH}}(0,0,1) \; =\; (r_0, \, l_0, \, \chi_0).$$ 
  For $r_0>0$ there is nothing to prove. 
  For $r_0<0$ we simply set $\Psi_1:=\Phi\circ [1]$ in order to make $r_0$ positive. 
  Suppose now $r_0=0$.
  If $\chi_0\neq 0$, then it is enough to set $\Psi_1:=\sS_B \circ \Phi$ for $\chi_0>0$, and $\Psi_1:=\sS_B\circ \Phi[1]$ if $\chi_0<0$. 
  Assume now $r_0=\chi_0=0$ and let $\Phi^{\rm CH}(1,0,0):=w_0=(s_0,b_0,\xi_0)$. 
  By noting that $v(\sO_x)=(0,0,1)$ for any point $x\in B$, and 
   $v(\sO_B)=(1,0,0)$, by \eqref{isometry} we find that 
  $\langle v_0, w_0\rangle_B=l_0\cdot b_0=1$. 
 Moreover, if $B_0$ is a line bundle with class $b_0$,
  then the composition $\big(T_B(B_0)\circ \Phi\big)^{\rm CH}$ sends 
  $(0,0,1)$ to  $(0, l_0,  1) $ as showed in \eqref{actionH}. 
  So for this case we set $\Psi_1:=\sS_B\circ T_B(B_0)\circ \Phi$.
%
%

We now show that there exists an equivalence $\Psi_2:\rd(A)\rightarrow \rd(C)$ with $C\in \{B,\widehat B\}$ such that in the vector
$v_2:=\Psi_2^{\rm CH}(0,0,1)=(r_2,l_2,\chi_2)$ 
we have $r_2>0$ with $\chi_2$ coprime with $r_2$. 
Let $w_1:=\Psi^{\rm CH}_1(1,0,0)=(s_1,b_1,\xi_1)$ so that $\langle v_1, w_1\rangle=1$ and 
\begin{equation}\label{II}
I \, - \, r_1 \, \xi_1 \, -\, \chi_1 \, s_1 \; = \; 1\quad \mbox{ where }\quad I\; := \; l_1\cdot b_1.
\end{equation}
Let $B_1$ be a line bundle such that its numerical class is $b_1,$
and consider the composition $\big( T_C(B_1^{\otimes n})\circ \Psi_1\big)^{\rm CH}$ with $n\in \rz_{>0}$. Again by \eqref{actionH}
this equivalence sends
$(0,0,1)$ to 
$$v_2\; := \;(r_2, \, l_2, \, \chi_2)\; = \: \Big(r_1, \, l_1 \, +\, r_1 \, n\, b_1, \, \chi_1 \, + \, n \, I\, +\, r_1 \,n^2  \,
\frac{b_1^2}{2}\Big).$$
Therefore the first entry of $v_2$ is positive and moreover there is no divisor of $r_1$ that divides both $I$ and $\chi'$ (see \eqref{II}). 
Suppose now that there is a prime divisor $q$ of $r_1$ that divides $\chi_1$ but not $I$. 
In this case we can choose the integer $n$ to be coprime with $q$ that in turns makes $\chi_2$ coprime with $q$. Since there are infinitely many
choices for such $n$, it follows that $\chi_2$ is coprime with $r_2=r_1$.
Suppose now that there is a prime 
divisor $q$ of $r_1$ that divides $I$ but not $\chi_1$. Then it is easy to see that $\chi_2$ is again coprime with $r_2$ (in fact 
we can choose any integer $n$).
Suppose now that both $\chi_1$ and $I$ are coprime with $r_2$. We can choose $n$ so that $\chi_1+n I \neq 0$ modulo any prime divisor 
$q$ of $r_1$. With this choice of $n$, it is easy to verify that $\chi_2$ is relatively prime with $r_2$.
Therefore we set $\Psi_2:=T_C(B_1^{\otimes n})\circ \Psi_1$.

Let now $\Theta$ be an ample line bundle with class $\theta \in {\rm CH}^1(C)$ and consider the equivalence 
$\Psi_3:=T_C(\Theta^{\otimes (r_1\,d)})\circ \Psi_2$ where $d$ is a sufficiently large integer.
Then $\Psi_3^{\rm CH}$ sends $(0,0,1)$ to 
$$v_3 \; := \; \Big(r_1, \, l_1 \, + \, r_1 \, n \, b_1 \, + \, r_1^2 \, d\,  \theta, \, \chi_1 \, + \, n \, I\, 
+\, r_1 \,n^2 \, \frac{b_1^2}{2} \, +\, r_1 \,d \,\big(\theta \cdot (l_1 \, + \, r_1 \,n \,b_1) \big) \, +\, r_1^3 \, d^2 \, 
\frac{\theta^2}{2}\Big).$$
We notice that, for $d$ large enough, the second component of $v_3$ is an ample class, 
and that the third component of $v_3$ is congruent to $\chi_2$ 
modulo $r_1$. Therefore it is itself coprime with $r_1$. In conclusion, the equivalence we are looking for is simply $\Psi_3$.
\end{proof} 
\begin{theorem}\label{moduli}
Let $A$ be an abelian surface and let $\Phi:\rd(A)\rightarrow \rd(Y)$ 
be an equivalence of derived categories of smooth projective varieties. 
Then $Y$ is an abelian surface. Moreover, $A$ is isomorphic to a moduli space of Gieseker-stable 
sheaves on either $Y$, or its dual $\widehat{Y}$, according to whether the rank component of  
$\Phi^{\rm CH}(0,0,1)$ is non-zero, or zero respectively. 

\end{theorem}

\begin{proof}
By general theory $Y$ is a smooth surface with trivial canonical bundle. Let $l\neq p$ be a prime and consider the $l$-adic cohomology groups $H^i_{\acute{e}t} (Y,\mathbf{Q}_l)$. By \cite[Lemma 3.1]{Ho1} the equivalence $\Phi$ induces an isomorphism
 $$H^1_{\acute{e}t} (Y,\mathbf{Q}_l) \; \oplus \; H^3_{\acute{e}t} (Y,\mathbf{Q}_l)(2) \; \simeq \; H^1_{\acute{e}t} (A,\mathbf{Q}_l)
 \; \oplus\;  H^3_{\acute{e}t} (A,\mathbf{Q}_l)(2)$$
 which leads to the equality of Betti numbers:
 $$b_1(Y) \;= \;b_1(A) \; = \; 4.$$
In positive characteristic the fact that a smooth surface has trivial canonical bundle and $b_1=4$ 
imply that $Y$ is isomorphic to an abelian surface (see \cite[Section 7]{Li}).
 
 By Proposition \ref{RelPrime3} there exists an equivalence $\Psi:\rd(A)\rightarrow \rd(C)$ with $C\in \{Y,\widehat{Y}\}$ such that 
 in the vector 
 $$v\; := \;  \Psi^{\rm CH}(0,0,1) \; = \; (r, \, l, \, \chi)$$ the rank component $r$ is positive, the class 
 $l$ is ample, and $\chi$ 
 is coprime with $r$. 
Moreover, by looking at the proof of Proposition \ref{RelPrime3}, it is the case that $C=Y$ if the rank component of  
$\Phi^{\rm CH}(0,0,1)$ is non-zero, and $C=\widehat{Y}$ otherwise. Now let $L$ be a line bundle on $C$ whose class is $l$ and 
set $v_l=(r,l,\chi)\in {\rm CH}^*(C)$.
 Then by Theorem \ref{modulispace} there is a universal family $\sU$ on
  $\sM_l(v_l)\times C$ inducing an equivalence 
 $$\Phi_{\sU}\,: \,\rd\big(\sM_l(v_l) \big)\rightarrow \rd(C)$$ of derived categories satisfying
 $\Phi^{\rm CH}_{\sU}(0,0,1) =  v.$
As the composition $\big(\Phi_{\sU}^{-1}\circ \Psi\big)^{\rm CH}$ sends 
$(0,0,1)$ to $(0,0,1)$, by Proposition \ref{filteredequivalence} we get $A\simeq \sM_l(v_l)$. 
 
\end{proof}

\section{FM partners of canonical covers of hyperelliptic surfaces}\label{FMhyp}
We denote the set of Fourier--Mukai partners of a smooth projective variety $X$ by 
$${\rm FM}(X) \; := \; \{\, Y \; | \; Y\mbox{ is a smooth projective variety with \,} \rd(Y)\simeq \rd(X)\}_{/\simeq}.$$
We say that ${\rm FM}(X)$ is trivial if ${\rm FM}(X)=\{X\}$.
In the case of an abelian variety $A$, we say that its set of Fourier--Mukai partners 
is trivial if ${\rm FM}(A)\subset \{A,\widehat{A}\}$.

A \emph{hyperelliptic surface} over an algebraically closed field $k$ of positive characteristic $p>3$
is a smooth projective minimal surface $X$ with $K_X\equiv 0$, $b_2(X)=2$, and such that each fiber of the Albanese map is a 
smooth elliptic curve (\emph{cf}. \cite[\S 10]{Ba}). These surfaces can be described as quotients $(E\times F)/G$ 
of two elliptic curves $E$ and $F$ by a finite group $G$. The group $G$ acts on $E$ by translations, and on $F$ in a way such that $F/G\simeq \rp^1$. 
Moreover, there are only a finite number of possibilities for the action of $G$ on $E\times F$, which have been classified by
Bagnera--De Franchis \cite[10.27]{Ba}. 

By \cite[\S 9.3]{Ba} the order $n$ of the canonical bundle of $X$ is finite with $n=2,3,4,6$.
Therefore we can consider the \emph{canonical cover} $\pi:\widetilde{X}\rightarrow X$ of the surface $X$ 
which is the \'{e}tale cyclic cover associated to the canonical bundle $\omega_X$. 
The degree of $\pi$ is the order $n$ of $\omega_X$, and in addition $\pi$ comes equipped with an action of the cyclic group 
that realizes $X$ as the quotient $\widetilde{X}/(\rz/n\rz)$.
By looking at the Bagnera--De Franchis' list \cite[10.27]{Ba}, the canonical cover $\widetilde{X}$ 
of an arbitrary hyperelliptic surface $X=(E\times F)/G$ is an abelian surface that sits inside a tower of surfaces
$$E\times F \; \stackrel{\pi'}{\longrightarrow} \; \widetilde{X}\; \stackrel{\pi}{\longrightarrow}\; X,$$ where 
$\pi'$ is an \'{e}tale cyclic cover of degree one, two, or three. Moreover, if $\pi'$ has degree three, then $F$ admits an automorphism of groups 
of order three and has $j$-invariant equals to zero.
Therefore the dual morphism $\widehat{\pi'}$ realizes the dual of $\widetilde{X}$ either as 
the product $E\times F$, or as a degree two \'{e}tale cyclic cover of $E\times F$, or as 
a degree three \'{e}tale cyclic cover 
of $E\times F$ such that $F$ has an automorphism of groups of order three. 

\subsection{The work of Sosna}
In \cite[Theorem 1.1]{So} the author proves that the set of Fourier--Mukai partners of the canonical cover of a complex hyperelliptic 
surface is trivial. By using Bagnera--De Franchis' classification, 
Sosna's theorem boils down to proving the following result concerning derived equivalences of special abelian surfaces.

\begin{theorem}[Sosna]\label{MTgCori}
 Let $E$ and $F$ be complex elliptic curves and $A$ be a complex abelian surface.
 Then ${\rm FM}(E\times F)$ is trivial. Moreover, if $E\times F\rightarrow A$ is a degree two \'{e}tale cyclic cover, then
 ${\rm FM}(A)$ is trivial. Finally, the same conclusion holds if $E\times F\rightarrow A$ is a degree three \'{e}tale cyclic cover and ${\rm rk \, NS}(A) \in  \{2,4\}$.
\end{theorem}

\noindent In view of Theorem \ref{etalecover} we prefer to work with \'{e}tale covers instead of quotients. Thus we reformulate Sosna's theorem in 
the following version.

\begin{prop}\label{MTgC}
 Let $E$ and $F$ be two complex elliptic curves. 
 Then ${\rm FM}(E\times F)$ is trivial.
  Moreover, if $A$ is a degree two \'{e}tale cyclic cover of $E\times F$, then ${\rm FM}(A)$ is trivial as well. 
  Finally, the same conclusion holds if $A$ is a degree three \'{e}tale cyclic cover of $E\times F$ and
  ${\rm rk \, NS}(A) \in  \{2,4\}$. 
\end{prop}
\begin{proof}
 If $A\rightarrow E\times F$ is a cover of degree one, two, or three, then the dual isogeny $E\times F\rightarrow \widehat A$ realizes 
 $\widehat{A}$ as a quotient 
 of two elliptic curves. Then by Theorem \ref{MTgCori} we conclude that ${\rm FM}(\widehat{A})$ is trivial. As ${\rm FM}(A)={\rm FM}(\widehat{A})$ and 
 ${\rm rk\, NS}(A)={\rm rk \, NS}(\widehat A)$, the proposition follows at once.
\end{proof}

\noindent As an application of Proposition \ref{MTgC}, we deduce some further finitiness results that 
will be useful towards the proof of Theorem \ref{main1}.

\begin{prop}\label{MTgC2}
If $\varphi \, : \, A\rightarrow E\times F$ is an isogeny 
 with ${\rm deg} \, \varphi=8$ and ${\rm exp} \, \varphi=2$, then 
 ${\rm FM}(A)$ is trivial. The same conclusion holds if ${\rm rk\, NS}(A) \in  \{2,4\}$, 
 ${\rm deg}\, \varphi=27$ and ${\rm exp}\,\varphi = 3$.
\end{prop}
\begin{proof}
 We show that the dual abelian variety $\widehat{A}$ satisfies the hypotheses of Proposition \ref{MTgC}. 
 The result will follows as ${\rm FM}(A)={\rm FM}(\widehat{A})$. Let $q$ be either $2$ or $3$
  and consider an isogeny $\psi:E\times F\rightarrow A$ of exponent $q$ such that $\psi\circ \varphi =  q_A$. 
  As ${\rm deg} \, q_A=q^4$ and ${\rm deg}\, \varphi=q^3$, we deduce that ${\rm deg}\,\psi=q$. Hence the dual isogeny 
  $\widehat{\psi}$ is a cyclic cover of $E\times F$ of order $q$. The second statement follows as
  ${\rm rk\, NS}(A) =  {\rm rk\, NS}(\widehat{A})$. 
 \end{proof}

\subsection{Strategy of the proof of Theorem \ref{main1}}

Since an abelian surface and its dual have the same Fourier--Mukai partners, the following theorem in particular proves Theorem \ref{main1}. 

\begin{theorem}\label{MTg}
Let $E$ and $F$ be elliptic curves over an algebraically closed field of characteristic $p>0$. Then 
${\rm FM}(E\times F)$ is trivial. Moreover, if $A$
is a degree two \'{e}tale cyclic cover over $E\times F$ and $p>2$, then ${\rm FM}(A)$ is trivial as well.
Finally, if $F$ admits an automorphism of groups of order three, $A$ is 
a degree three \'{e}tale cyclic cover over $E\times F$ and $p>3$, then ${\rm FM}(A)$ is again trivial.
\end{theorem}
\noindent In order to prove the previous result, we will consider the following set of hypotheses
\begin{set}\label{sett}
We denote by $E$ and $F$ two elliptic curves over an algebraically closed field $k$ of characteristic $p>0$. 
Moreover we set $\nu:A\rightarrow E\times F$ to be either an isomorphism of abelian surfaces, 
or an \'{e}tale cyclic cover of degree $d_{\nu}=2,3$ (as in the hypotheses of Theorem \ref{MTg}). Finally assume that $p>{\rm deg}\, \nu$.
\end{set}

\begin{rmk}\label{dualiso}
Since the exponent of an isogeny divides its degree, the exponent of the isogeny $\nu$ of Setting \ref{sett} is either 
one if $\nu$ is an isomorphism, 
or $d_{\nu}$ otherwise. Let now $\mu:E\times F\rightarrow A$ be an isogeny of exponent $d_{\nu}$ such that 
$\mu\circ \nu=(d_{\nu})_A$. Then the dual isogeny
$\widehat{\mu}:\widehat{A}\rightarrow E\times F$ is either 
an isomorphism, or else its degree and exponent satisfy $({\rm deg}\, \widehat{\mu},{\rm exp} \, \widehat{\mu}) \, = \, (d_{\nu}^3,\,d_{\nu}).$
\end{rmk}

\noindent As an application of Theorem \ref{etalecover} we deduce that 
both the isogenies $\nu$ and $\widehat{\mu}$ of Setting \ref{sett} and Remark \ref{dualiso} lift 
to the ring of Witt vectors. In the following result we check that their degrees and exponents remain unchanged when passing from the special fiber to the 
general fiber.

\begin{prop}\label{liftetale}
 Let $E$ and $F$ be elliptic curves and $\varphi:A\rightarrow E\times F$ be an \'{e}tale 
 isogeny of abelian surfaces. 
 Let $q\neq p$ be a prime integer and assume that either $\varphi$ is an isomorphism, or ${\rm deg} \, \varphi =q$, or
 ${\rm deg} \, \varphi=q^3$ and ${\rm exp} \, \varphi=q$.
 If $\sE\rightarrow W$ and $\sF\rightarrow W$ are projective lifts of $\sE$ and $\sF$ over a finite ramified extension of
 the ring of Witt vectors respectively, then 
there exist a projective lift $\sA\rightarrow W$ of $A$ and an isogeny $\varphi_W:\sA\rightarrow \sE \times_W\sF$ 
such that $\varphi_W$ lifts $\varphi$ and its restriction 
$\varphi_{\eta}:A_{\eta}\rightarrow E_{\eta}\times F_{\eta}$ to the geometric general fibers is an 
isogeny with ${\rm deg} \, \varphi_{\eta}={\rm deg}\, \varphi$ and ${\rm exp} \, \varphi_{\eta}={\rm exp}\, \varphi$
\end{prop}

\begin{proof}
 By Theorem \ref{etalecover} there is a projective lift $\sA\rightarrow W$ of $A$ 
 and an \'{e}tale cover $\varphi_W:\sA\rightarrow \sE\times_{W}\sF$ that specializes to $\varphi$. 
 Up to composing with a translation of $\sE\times_{W}\sF$, we 
 can suppose that $\varphi_W$ is a homomorphism of groups. Hence the restriction of 
 $\varphi_W$ to the geometric generic fiber
 of $\sA$ is an isogeny $\varphi_{\eta}:A_{\eta}\rightarrow E_{\eta}\times F_{\eta}$ such that ${\rm deg}\,\varphi_{\eta}={\rm deg}\, \varphi$.
To see this we notice that the kernel $\sK$ of $\varphi_W$ is a finite \'{e}tale group over $W$ and moreover,
as $\varphi$ is separable, we have $${\rm deg}\, \varphi_{\eta}=|{\rm ker}\,\varphi_{\eta}|=  |\sK_{\eta}|=|\sK_{k}|
=|{\rm ker}\, \varphi|={\rm deg}\, \varphi$$ where $\sK_k$ is the closed fiber and $\sK_{\eta}$ is the geometric generic fiber.
In particular this takes care of the case when the degree $\varphi$ is either one or $q$.

Suppose now that ${\rm deg}\,\varphi=q^3$ and ${\rm exp} \, \varphi=q$. 
We only need to show that ${\rm exp}\, \varphi_{\eta}=q$. Let $\psi:E\times F\rightarrow A$ be an isogeny such that 
$\varphi \circ \psi=q_{E\times F}$ and let
$\psi_W:\mathcal{X}\rightarrow \sA$ be a lift of $\psi$ as in Theorem \ref{etalecover} so that 
$\big(\varphi_{W}\circ \psi_{W}\big)_{|\mathcal{X}_{k}}=\varphi\circ \psi=q_{E\times F}$.
As the multiplication-by-$q$ map $q_W:\sE\times_{W}\sF \rightarrow \sE\times_{W}\sF$ 
lifts $q_{E\times F}$ as well, we conclude that $\mathcal{X}\simeq \sE\times_{W}\sF$ and $\varphi_W\circ \psi_W=q_W$ 
By restricting to the geometric general fibers, we find that
$\varphi_{\eta}\circ \psi_{\eta} = q_{(E\times F)_{\eta}}$ is the multiplication-by-$q$-map on the geometric general fiber $(E\times F)_{\eta}$.
Therefore the 
exponent of $\varphi_{\eta}$ is either one or $q$, but if ${\rm exp}\,\varphi_{\eta}=1$, then both $\varphi_{\eta}$ and $\varphi$ 
are isomorphisms, which is excluded by hypothesis.
  
\end{proof}

%
%
%
%
%
\noindent We will deduce Theorem \ref{MTg} from the following technical proposition.
\begin{prop}\label{cond&stat}
Assume Setting \ref{sett} and
let $\Phi:\rd(B)\rightarrow \rd(A)$ be an equivalence of derived categories of abelian surfaces. 
Suppose that there exists an equivalence $\Psi:\rd(B)\rightarrow 
 \rd(C)$ with $C\in\{A,\widehat{A} \}$ such that the Mukai vector
 $$v\, := \, (r, \, l, \, \chi) \, = \, \Psi^{\rm CH}(0, \, 0, \, 1)$$ satisfies the following conditions:
 \begin{enumerate}
  \item[$(E_1).$] $r$ is positive;\\
  \item[$(E_2).$] the class $l\,\in {\rm CH}^1(C)$ is is ample;\\
  \item[$(E_3).$] $\chi$ is coprime with $r$.
 \end{enumerate}
Set now $\lambda=\nu$ if $C=A$, and $\lambda=\widehat{\mu}$ otherwise (see Remark \ref{dualiso}). In addition assume that there exist
projective lifts $\sE\rightarrow W$ and $\sF\rightarrow W$ of $E$ and $F$ over a finite ramified extension of the ring of Witt vectors respectively
such that the following conditions hold: 
 \begin{enumerate}
  \item[$(A_1)$.] If we denote by 
  $\lambda_W:\sC\rightarrow \sE\times_W\sF$ the lift of $\lambda$ 
  determined by Proposition \ref{liftetale}, and by $L$ an ample line bundle on $C$,  
  then $L$ 
  lies in the image of the restriction map
  $$\rho \; : \; {\rm Pic}(\sC) \, \rightarrow \, {\rm Pic}(C).$$
  \item[$(A_2)$.] If ${\rm deg} \, \nu  =  3$ and $F$ admits an automorphism of groups of order three, 
  then ${\rm rk \, CH}^1(\sC_{\eta})\in \{2,4\}$ where $\sC_{\eta}$ is the geometric generic fiber of $\mathcal{C}\rightarrow W$.
   \end{enumerate}
   \noindent Then Theorem \ref{MTg} holds.
   \end{prop}

\begin{proof}
 \emph{Step 1.} We first prove that there exists a projective lift $\sB\rightarrow W$ of $B$ such that the geometric generic fiber 
 $\sB_{\eta}$ is derived equivalent to $\sC_{\eta}$. Let $L$ be an ample line bundle on $C$ with class $l$.
 Then Theorem \ref{moduli} implies that $B$ is isomorphic to a moduli space 
 $\sM_{l}(v_l)$ of Gieseker-stable shaves with Mukai vector 
 $v_l=(r,l,\chi)\in {\rm CH}^*(C)$. 
 Consider now a preimage $\widetilde{L}$ of $L$ under $\rho$ as in $(A_1)$ and
 the relative moduli space $$\sM_{\sC/W}(\widetilde{v}_l)\rightarrow W,$$ where $\widetilde{v}_l=(r,\widetilde{l},\chi)$ and $\widetilde{l}$ is the class of $\widetilde{L}$.
 As discussed in \S\ref{modulispacech}, this is a projective lift of $B$ and the geometric generic fiber $\sM_{\eta}$  
is a moduli space of Gieseker-stable sheaves on $\sC_{\eta}$ with Mukai vector 
$v_{\eta}=(r,\widetilde{l}_{|\sC_{\eta}},\chi)$.
Therefore as discussed in Theorem \ref{modulispace}, the condition $(E3)$ implies that 
there exists a universal family $\sU_{\eta}$ on $\sM_{\eta}\times \sC_{\eta}$ inducing an equivalence 
$\Phi_{\sU}:\rd(\sC_{\eta})\rightarrow \rd(M_{\eta})$ 
where $M_{\eta}$ is an irreducible component of $\sM_{\eta}$.

\noindent \emph{Step 2.} The equivalence $\Phi_{\sU}$ shows that $M_{\eta}$ is an abelian surface.
Now we prove that under the assumptions of Theorem \ref{MTg} the abelian surface $\sC_{\eta}$ is isomorphic to either
$M_{\eta}$ or its dual $\widehat{M_{\eta}}$.
By Lefschetz's principle we can suppose that the abelian surface $C$ is defined over a subfield of the complex numbers $\rc$
and therefore that 
$\sC_{\eta}$ is defined over $\rc$.
Suppose first that $\nu:A\rightarrow E\times F$ is an isomorphism. 
Then both $\lambda$ and $\lambda_W$ are isomorphisms and therefore so is the restriction 
$\lambda_{\eta}:\sC_{\eta} \rightarrow \sE_{\eta}\times \sF_{\eta}$ of $\lambda_W$ to the geometric generic fibers.
As a product of elliptic curves has no non-trivial Fourier--Mukai partners (\emph{cf}. Theorem \ref{MTgCori}), 
we deduce then an isomorphism $\sC_{\eta}\simeq M_{\eta}$.

Suppose now that $\nu:A\rightarrow E\times F$ is a degree two cyclic cover. 
By Remark \ref{dualiso} we have $({\rm deg} \, \lambda,{\rm exp} \, \lambda)\in
\{(2,2),(8,2)\}$
and by Proposition \ref{liftetale} 
$(\deg \, \lambda_{\eta},
 {\rm exp} \, \lambda_{\eta})=(\deg{\lambda},{\rm exp}\, \lambda)$. 
 Therefore by Proposition \ref{MTgC} and Proposition \ref{MTgC2} we deduce  that
 either $\sC_{\eta}\simeq M_{\eta}$ or $\sC_{\eta}\simeq \widehat {M_{\eta}}$.
  The case when $\nu$ has degree three follows similarly by involving the condition $(A_2)$.
 
 \noindent \emph{Step 3.} To conclude the proof we use the argument of \cite[Lemma 6.5]{LO} (based on a result of Matsusaka--Mumford)
 in order to prove that the isomorphism 
 $\sC_{\eta} \simeq M_{\eta}$ (\emph{resp.} $\sC_{\eta}\simeq \widehat{M_{\eta}}$) between the geometric generic fibers of the two liftings
 induces an isomorphism $C\simeq B$ (\emph{resp.} $C\simeq \widehat{B}$) between the closed fibers. This immediately yields that 
 either $B\simeq A$ or $B\simeq \widehat{A}$, and hence that ${\rm FM}(A)$ is trivial.
 
 \end{proof}

\section{Finding a suitable equivalence}\label{finopp}
In this section we finish the proof of Theorem \ref{main1}. According to Proposition \ref{cond&stat}, we only need to verify its hypotheses. 
We work under the hypotheses of Setting \ref{sett} and assume that the abelian surface $B$ is a Fourier--Mukai partner of $A$. 
In the following we will distinguish two cases: $(a)$ at least one of the two elliptic curves $E$ or $F$ is ordinary, 
 and $(b)$ both $E$ and $F$ are supersingular.
  
\subsection{The case where one of the two curves is ordinary}
The following two propositions show the existence of an equivalence $\Psi:\rd(B)\rightarrow \rd(C)$ satisfying the hypotheses 
$(E_1)$, $(E_2)$ and $(E_3)$ of 
Proposition \ref{cond&stat}.

\begin{prop}\label{RelPrime}
Let $A$ and $B$ be abelian surfaces and $\Phi :  \rd(B) \rightarrow \rd(A)$ be an equivalence of derived categories. 
Moreover fix two distinct primes $p_1$ and 
$p_2$. Then there exists an equivalence $\Psi  : \rd(B) \rightarrow \rd(C)$ with $C\in \{A,\, \widehat{A}\}$ such that the vector 
$$ (r, \, l, \, \chi) \; := \; \Psi^{\rm CH}(0, \,0, \,1)$$ 
satisfies one the two following statements:
\begin{enumerate}
 \item $r$ is relatively prime with both $p_1$ and $p_2$;\\
 \item either $p_1$ divides $r$ but not $\chi$ and $p_2$ divides $\chi$ but not $r$, or \emph{viceversa}.
\end{enumerate}
\end{prop}

\begin{proof} 
 Set $$v_0 \; := \; (r_0, \, l_0, \, \chi_0) \; = \; \Phi^{\rm CH}(0, \,0, \,1)$$
  $$w_0\; := \; (s_0, \, h_0, \, \xi_0) \; = \; \Phi^{\rm CH}(1, \,0, \,0).$$
   Since $\Phi$ is an equivalence, by \eqref{isometry} we find that 
 \begin{equation}\label{vw}
  1 \, = \, \langle v_0, w_0\rangle_A \, = \, I \, - \, s_0 \, \chi_0 \, - \, r_0 \, \xi_0,\quad  \mbox{ where }\quad I \; := \; l_0\cdot h_0.
  \end{equation}
 Let $H_0$ be a line bundle on $A$ such that its class in the N\'{e}ron--Severi group is $h_0$. Therefore 
 at the level of numerical Chow rings the equivalence 
 $\Phi_n :=  T_A(H_0^{\otimes n}) \circ \Phi  :  \rd(B)  \rightarrow  \rd(A)$ ($n\in \rz_{>0}$) sends $(0,0,1)$ to
  $$v_n \, := \, \Phi_n^{\rm CH}(0,0,1) \, = \, \big(r_0,\, l_0\,+\,r_0\,n\,h_0, \,\chi_n \big)$$
 where 
 \begin{equation}\label{chinew}
 \chi_n\, := \, \chi_0 \, + \, n\,I\, + \, r_0 \, n^2 \,\frac{h_0^2}{2}.
 \end{equation}
\noindent  We divide the proof in five cases.
 
 \noindent \emph{Case I:} Suppose that neither $p_1$ nor $p_2$ divides $\chi_0$. In this case the equivalence $\Psi$ is given by the composition
 $\sS_A \circ \Phi\, : \, \rd(B)\rightarrow \rd(\widehat{A})$.
 
 \noindent \emph{Case II:} Suppose that both $p_1$ and $p_2$ divides both $r_0$ and $\chi_0$.
  By \eqref{vw} we see that $I$ is relatively prime with $p_1$ and $p_2$ as well.
 Choose now a positive integer $n$ coprime with both $p_1$ and $p_2$. 
 Therefore by looking at the 
 definition \eqref{chinew} of $\chi_n$, this immediately implies that $\chi_n$ is relatively prime to both $p_1$ and $p_2$. 
 We conclude then as in Case I.
 
 \noindent \emph{Case III:} Suppose that both $p_1$ and $p_2$ divides $r_0$, and that precisely one of them, say $p_1$, divides $\chi_0$.
 We choose a generic positive integer $n$ such that $n$ is relatively prime to both $p_1$ and $p_2$.
 By \eqref{vw} $I$ is relatively prime to $p_1$, and by \eqref{chinew} $p_1$ does not divide $\chi_n$. Moreover, again by \eqref{chinew} and the 
 fact that $n$ is general, we can suppose that $p_2$ does not divide $\chi_n$ as well. We then set $\Psi:= \sS_A \circ \Phi_n$.
 
 \noindent \emph{Case IV:} Suppose that both $p_1$ and $p_2$ divides $\chi_0$ and that precisely one of them, say $p_1$, divides $r_0$.
In this case we proceed as in Case III by considering the composition $\sS_A\circ \Phi$ in place of $\Phi$.

\noindent \emph{Case V:} Suppose that one of the primes, say $p_1$, divides both $r_0$ and $\chi_0$, but $p_2$ divide nor $r_0$ neither 
$\chi_0$. Let $n=p_2$ and consider $\Phi_n$. By \eqref{vw} $p_1$ does not divide $I$, and hence $p_1$ does not divide $\chi_n$. 
Moreover, by our choice of $n$, we have that $p_2$ does not divide $\chi_n$ as well. We conclude then as in Case I.
 \end{proof}

\begin{prop}\label{RelPrime2}
 Assume the assumptions of Setting \ref{sett} and let $\Phi: \rd(B)\rightarrow \rd(A)$ be an equivalence of triangulated categories. 
 Moreover assume that ${\rm deg}\,\nu\geq 2$.
 Then
 there exists an equivalence 
 $\Xi : \rd(B) \rightarrow  \rd(C)$ with $C\in \{A, \, \widehat{A}\}$  
 such that the vector 
 $$(r, \, l, \, \chi) \; := \; \Xi^{\rm CH}(0,\,0,\,1)$$ satisfies:
 \begin{enumerate}
  \item $r$ is positive and relative prime with $p$;\\
  \item the class $l \in {\rm CH}^1(C)$ is ample;\\
  \item $\chi$ is relative prime with $r$.
 \end{enumerate}
 Set now $\lambda=\nu$ if $C=A$, and $\lambda=\widehat \mu$ otherwise. Then 
 the class $l$ is the pull-back of some ample class in ${\rm CH}^1(E\times F)$ via $\lambda$. 
 \end{prop}
 
\begin{proof}
 By using Proposition \ref{RelPrime} with $p_1=p$ and $p_2=\deg \, \nu$, we can find an equivalence
 $\Psi:\rd(B)\rightarrow \rd(C)$ with $C\in \{A,\widehat{A}\}$ such that the vector $$v_0:=(r_0, \, l_0, \, \chi_0)=\Psi^{\rm CH}(0, \,0, \,1)$$ 
 satisfies one 
 of the two following conditions: $(a)$ $r_0$ is relatively prime to both $p_1$ and $p_2$, or $(b)$ one of the primes $p_1$ and $p_2$ 
 divides $r_0$ but not $\chi_0$, while the other divides $\chi_0$ but not $r_0$.
 Set $\lambda=\nu$ if $C=A$ and $\lambda =\widehat{\mu}$ otherwise.
 We claim that there exists an equivalence $\Xi_1:\rd(B)\rightarrow \rd(C)$ such that 
 the vector $$v_1:=(r_1, \, \lambda^*l_1, \, \chi_1)\, = \, \Xi_1^{\rm CH}(0, \, 0, \,1)$$
 satisfies: $r_1$ is positive and relatively prime with $p$, and $l_1\in {\rm CH}^1(E\times F)$. 
  In order to prove the claim, we distinguish the two cases $(a)$ and $(b)$ mentioned above. 
 Suppose first case $(a)$. As $r_0$ is not zero, we can make it positive by composing with the shift functor, if necessary.
 Let $L$ be a line bundle representing $l_0$ and $m$ be a positive integer such that $p_2^2$ divides $(r_0 \, m+1)$. Hence we can write $1+r_0 \, m=
 p_2^2 \, u$ for some integer $u$ and we consider the composition $T_C(L^{\otimes m})\circ \Psi$. 
 Then 
 $$\big(T_C(L^{\otimes m})\circ \Psi \big)^{\rm CH}(0,0,1) \; = \; \big(r_0, \, (p_2^2 \, u)\, l_0, \,\chi_1\big)$$ for some integer $\chi_1$.
 By Proposition \ref{Neron} there exists a class $l_1\in {\rm CH}^1(E\times F)$ such that $(p_2^2 \, u) \, l_0=\lambda^* (u\, l_1)$. 
 This proves the claim in case 
 $(a)$ as we can set $\Xi_1:=T_C(L^{\otimes m})\circ \Psi$ and $v_1:=(r_0, \, (p_2^2 \,u)\, l_0, \,\chi_1)$.
 
 We now suppose case $(b)$. If necessary we replace $\Psi$ with $\sS_C\circ \Psi$ in order to make $r_0$ 
 divisible by $p_1=p$, but relatively prime with $p_2={\rm deg} \, \nu$. Let $L$ and $m$ be as above. Then the equivalence $\Psi_1 :=
 T_C(L^{\otimes m})\circ \Psi$ sends the vector $(0,0,1)$ to  
 $$\Psi_1^{\rm CH}(0, \, 0, \, 1)\; = \; \big(r_0, \, p_2^2\, u\, l_0, \, \chi_0 \, +\, m\, l_0^2\, +\, r_0\, m^2
 \frac{l_0^2}{2}\big) \, := \, 
 (r_0, \, p_2^2 \, l_2, \, \chi_2)$$ where $l_2$ is a class in ${\rm CH}^1(C)$.
 Since $\langle v_0,v_0 \rangle_C=0$, we have that $l_0^2=2r_0\chi_0$ and hence
 $$\chi_2\, = \, \chi_0 \, +\,2 \, m \, r_0\, \chi_0 \, +\, m^2\, r_0^2 \, \chi_0.$$ 
 As $p_1$ divides $r_0$, but does not divide $\chi_0$, it follows that 
 $$\chi_2 \; \equiv \; \chi_0 \; \not \equiv \; 0\; \quad (\mbox{mod }\,p).$$ 
 Hence the composition $\sS_C \circ \Psi_1$ sends $(0,0,1)$ to
 $$\big(\sS_C \circ \Psi_1\big)^{\rm CH}(0,0,1) \; = \; \big(\chi_2, \,p_2^2 \, \omega, \, r_0) \, = \, 
 (\chi_2, \, \lambda^*\omega', \,r_0\big)$$ for some elements 
 $\omega \in {\rm CH}^1(C)$ and $\omega'\in{\rm CH}^1(E\times F)$, and $p_2={\rm deg}\, \nu$ does not divide $r_0$ (here we used the fact that 
 $\sS_{C}^{\rm CH}$
 induces an homomorphism ${\rm CH}^1(C)\rightarrow {\rm CH}^1(\widehat{C})$; \emph{cf}. \S\ref{autoeq}). 
 If necessary we can make the integer $\chi_2$ positive by composing with the shift functor.
 This concludes the proof of the claim.
 
 Let now $w=(s, \, h, \,\xi)$ be the Mukai vector of $\Xi_1^{\rm CH}(1,0,0)$
 so that $\langle w, v_1\rangle_C=1$.
 Moreover let $H$ be a line bundle whose numerical class is $h$ and note that 
 \begin{equation}\label{rightn2}
  I\, - \, r_1 \, \xi \, - \chi_1 \, s \, = \, 1\quad \mbox{ where } \quad I:=l_1\cdot h.
 \end{equation}
 Choose a positive integer $n$ such that 
$$n \, p_2^2 \,I \, + \, \chi_1\; \not \equiv 0 \quad (\mbox{mod } q) \quad \mbox{ for every prime divisor }q\neq p_2 \mbox{ of }r_1 \mbox{ that does not 
divide }I.$$
Set now  $\Xi_2:=T_C(H^{\otimes (n \,p_2^2)})\circ \Xi_1$ and set 
 $(r_2,l_2,\chi_2):=\Xi_2^{\rm CH}(0,0,1)$.  With a simple calculation we find
 \begin{equation}\label{chi6}
 r_2 \; = \; r_1, \quad l_2\; = \; \lambda^* l_1 \, + \, r_1 \, n \, p_2^2 \, h,\quad 
 \chi_2 \; = \; \chi_1 \, + \, n \, p_2^2 \, I \, + \,r_1 \, n^2 \, p_2^4 \, \frac{h^2}{2}.
 \end{equation}
 We note that $r_2$ is not divisible by $p$, and moreover that by \eqref{chi6} $\chi_2$ is not divisible by any prime divisor 
 $q\neq p_2$ of $r_1$ that does not divide $I$.
 On the other hand, if a prime divisor $q\neq p_2$ divided both $r_1$ and $I$, then by \eqref{rightn2} it does not divide $\chi_1$, and hence 
 neither $\chi_2$. Finally we prove that $\chi_2$ is not divisible by $p_2$ in case $p_2$ divides $r_1$. 
 But this follows by the construction of 
 $\Xi_1$, and by noting that in the case $(b)$ earlier discussed, $p_2$ does not divide $r_0$.
 
 Consider now the composition 
 $\Xi_3 \, := \, T_C(\lambda^*\Theta^{\otimes (r_2 \,d)})\circ \Xi_2$ where $\Theta:=\sO_E(O_E)\boxtimes \sO_F(O_F)$ 
 and $d\gg 0$ is a positive integer.
 By a direct computation we have that 
 $\Xi_3^{\rm CH}$ sends the vector $(0,0,1)$ to 
 $$v_3 \; := \; (r_2, \, l_2 \, + \, r_2^2 \, d\, \lambda^*\theta,\, \chi_3) $$
 where $\theta$ is the numerical class of $\Theta$ and  
 $$\chi_3 \; := \; \chi_2 \, +\, r_2\, d \,(l_2 \cdot \lambda^*\theta )\, + \, r_2^3 \,d^2 \, \frac{(\lambda^*\theta)^2}{2}.$$
 By Lemma \ref{Neron} and \eqref{chi6} we can write $l_2=\lambda^*l_3$ for some class $l_3\in {\rm CH}^1(E\times F)$, and hence 
 $$\chi_3 \, \equiv \, 
 \chi_2 \, + \, r_2 \, d \, p_2^c\, (l_3\cdot \theta) \, + \, r_2^3 \, p_2^c\, d^2\,\frac{\theta^2}{2} \, \equiv \, \chi_2 \quad (\mbox{mod }\, p_2)$$
 as the isogeny $\lambda$ has degree either $p_2^c$ with $c=1$ if $C=A$, and $c=3$ otherwise. 
 Therefore the first component of $v_3$ is still positive 
 and relatively prime with $p$, while the second component is a pull-back of an ample class in ${\rm CH}^1(E\times F)$ (for $d\gg 0$). 
 Moreover $\chi_3 \, \equiv \, \chi_2 \,(\mbox{mod }\,r_2)$ and hence 
 it is still relative prime with $r_2=r_1$. Our desired equivalence is hence given by $\Xi:=\Xi_3$. 
  \end{proof}

  \begin{rmk}
   In Proposition \ref{RelPrime2} the conclusion that $r$ is coprime with $p$ does not play any role towards the verification of the hypotheses 
   of Proposition \ref{cond&stat}. However it will used in the supersingular case in order to construct an equivalence satisfying the hypotheses
   of Proposition \ref{cond&stat}.
  \end{rmk}

 \noindent In order to prove that the hypotheses $(A_1)$ and $(A_2)$ of Proposition \ref{cond&stat} hold as well, we first prove a 
 couple of auxiliary results.
 \begin{prop}\label{liftpic}
  Let $E$ and $F$ be ordinary elliptic curves. 
  Then there exist projective liftings $\sE\rightarrow W$ and $\sF\rightarrow W$ of $E$ and $F$ 
  over the ring of Witt vectors respectively such that the restriction morphism 
  ${\rm Pic}(\sE\times_W\sF) \rightarrow  {\rm Pic}(E\times F)$
  is surjective.
 \end{prop}
 
\begin{proof}
 The product $E\times F$ is an ordinary abelian surface so that we can consider its canonical lift $(\sY\rightarrow W, F_{\sY})$, by virtue of  
 Theorem \ref{canonicallift}. Moreover the restriction morphism ${\rm Pic}(\sY)\rightarrow {\rm Pic}(E\times F)$ is surjective. 
 However by using the universal property of the fiber product, it is immediate to show that $\sY\simeq \sE\times_W\sF$
 where $\sE\rightarrow W$ and $\sF\rightarrow W$ are the canonical lifts of $E$ and $F$ respectively.
\end{proof}
  
  \begin{prop}\label{liftpic2}
 If $E$ is an ordinary elliptic curve and $F$ is supersingular, then any projective lifts 
 $\sE\rightarrow W$ and $\sF\rightarrow W$ of 
 $E$ and $F$ respectively are such that the restriction morphism ${\rm Pic}(\sE\times_W \sF)\rightarrow {\rm Pic}(E\times F)$ is surjective.
 \end{prop}
 
 \begin{proof}
As ${\rm Hom}(E,F)=0$ we obtain an isomorphism ${\rm Pic}(E\times F)\simeq {\rm Pic}(E)\times {\rm Pic}(F)$.
Let now ${\rm pr}_E^*L_E\otimes {\rm pr}_F^*L_F$ be an arbitrary line bundle on $E\times F$.  
Since line bundles on curves lift, we can consider lifts $\mathcal{L}_{\sE}$ and $\mathcal{L}_{\sF}$ of $L_E$ and $L_F$  
respectively. Hence the line bundle ${\rm pr}_{\sE}^*\mathcal{L}_{\sE}\otimes {\rm pr}_{\sF}^*\mathcal{L}_{\sF}$ on $\sE\times_W\sF$
is a lift of ${\rm pr}_E^*L_E\otimes {\rm pr}_F^*L_F$.
 \end{proof}

\begin{prop}\label{propord}
Assume the assumptions of Setting \ref{sett} and let $\Phi:\rd(B)\rightarrow \rd(A)$ be an equivalence of derived categories of abelian surfaces.
If both $E$ and $F$ are ordinary elliptic curves,
 then the hypotheses of Proposition \ref{cond&stat} hold.
\end{prop}

\begin{proof}
By Propositions \ref{RelPrime3} and \ref{RelPrime2} there exists an equivalence 
$\Xi:\rd(B)\rightarrow \rd(C)$ with $C\in\{A,\widehat{A}\}$ such that the vector $\Xi^{\rm CH}(0,0,1)=(r,\lambda^*l,\chi)$ satisfies the 
hypotheses $(E_1)$, $(E_2)$ and $(E_3)$ of Proposition \ref{cond&stat}. 
Let $L$ be a line bundle on $E\times F$ with class $l\in {\rm CH}^1(E\times F)$.
By Propositions \ref{liftpic} there exist
projective lifts $\sE\rightarrow W$ and $\sF\rightarrow W$ of $E$ and $F$ to the ring of Witt vectors $W$ respectively,  and
a lift $\mathcal{L}\in {\rm Pic}(\sE\times_W\sF)$ of $L$. 
Let $\lambda_W:\mathcal{C}\rightarrow \sE\times _W\sF$ be the lift of $\lambda :C\rightarrow E\times F$ defined 
by Proposition \ref{liftetale}. It follows that $\lambda_W^*\mathcal{L}$ lifts $\lambda^*L$ which proves condition $(A_1)$. 
In order to prove $(A_2)$ we can assume 
that $F$ has an automorphism of groups of order three. 
Then by Theorem \ref{canonicallift} the geometric generic fiber $\sF_{\eta}$ 
admits an automorphism of groups of order three and hence ${\rm rk \, CH}^1 (\sE_{\eta}\times 
\sF_{\eta})\in \{2,4\}$ where $\sE_{\eta}$ is the geometric generic fiber of the lift $\sE\rightarrow W$.
\end{proof}

\begin{prop}
Assume the assumptions of Setting \ref{sett} and let $\Phi:\rd(B)\rightarrow \rd(A)$ be an equivalence of derived categories of abelian surfaces.
If $E$ is ordinary and $F$ is supersingular or \emph{viceversa},
 then the hypotheses of Proposition \ref{cond&stat} hold.
\end{prop}

\begin{proof}
By Propositions \ref{RelPrime3} and \ref{RelPrime2} there exists an equivalence 
$\Xi:\rd(B)\rightarrow \rd(C)$ with $C\in\{A,\widehat{A}\}$ such that the vector $\Xi^{\rm CH}(0,0,1)=(r,\lambda^*l,\chi)$ satisfies the 
hypotheses $(E_1)$, $(E_2)$ and $(E_3)$ of Proposition \ref{cond&stat}. 
 By Proposition \ref{liftpic2}
 we can choose general lifts $\sE$ and $\sF$ of $E$ and $F$ to $W$ respectively 
 and a line bundle $\mathcal{L}\in {\rm Pic}(\mathcal{C})$
such that the condition $(A_1)$ holds (see the argument of Proposition \ref{propord}).
In fact we choose $\sE$ and $\sF$ so that there are no non-trivial morphisms between the geometric 
generic fibers $\sE_{\eta}$ and $\sF_{\eta}$ of $\sE$ and $\sF$ respectively. 
It follows that ${\rm rk \, NS}(\sE_{\eta}\times \sF_{\eta})=2$ independently from the fact that $F$ admits an automorphism of groups 
of order three or not. As in Proposition \ref{propord} this immediately implies that ${\rm rk \, NS}(\sC_{\eta})=2$.
 \end{proof}

\subsection{Supersingular case}
We are going to prove that the hypotheses of Proposition \ref{cond&stat} hold when the elliptic curves $E$ and $F$ are both supersingular.
The following proposition proves the conditions $(E_1)$, $(E_2)$ and $(E_3)$.

\begin{prop}\label{chiandr2}
 Assume the hypotheses of Setting \ref{sett} with ${\rm deg}\,\nu\geq 2$ and assume that the elliptic curves $E$ and $F$ are supersingular. 
 Moreover let $\Phi:\rd(B)\rightarrow \rd(A)$ be an equivalence of derived categories of abelian surfaces.
 Then there exists an equivalence 
 $\Xi:\rd(B)\rightarrow \rd(C)$ with $C\in \{A,\widehat{A}\}$ such that the Mukai vector 
 $$(r, \, l, \, \chi) \; := \; \Xi^{\rm CH}(0, \,0, \,1)$$ satisfies the following conditions:
 \begin{enumerate}
  \item $r$ is positive and relatively prime with $p$;\\
  \item the class $l\in NS(C)$ is ample; \\
  \item $\chi$ is relatively prime with $r$.
 \end{enumerate}
Set now $\lambda=\nu$ if $C=A$, and $\lambda=\widehat \mu$ otherwise. Then $l=\lambda^*l'$ 
where $l'=l(\varphi,d_1,d_2)$ is the class of a line bundle on $E\times F$ with $\varphi$ an \'{e}tale isogeny.
 
 \end{prop}

\begin{proof}
 Proceeding as in the proof of Proposition \ref{RelPrime2}, we can construct an equivalence
 $\Psi_1  :  \rd(B)  \rightarrow  \rd(C)$ where 
 $C\in \{A,\widehat{A}\}$ such that the vector 
 $$v_1 \; := \; (r_1, \, \lambda^*l_1,\, \chi_1) \; = \; \Psi_1^{\rm CH}(0,0,1)$$ 
 satisfies: $r_1$ is relatively prime with $p$, and $l_1=l(\varphi,m_1,m_2)$ is the class of a line bundle on $E\times F$ for some morphism 
 $\varphi:F\rightarrow E$ and integers $m_1,m_2$ (see Proposition \ref{Picardprod}). We notice that we can assume $r_1$ positive by eventually 
 composing with the shift functor.
 Suppose first that $\varphi$ is either the constant morphism or a non-separable one.
 Let $t\neq p$ be a prime and let $\xi:F\rightarrow E$ be a separable isogeny of degree $t^k$ ($k\in \rz_{>0}$) determined by 
 Proposition \ref{koh}.
 Consider now the composition of equivalences $\Psi_2:=T_C\big(\lambda^*(1_E\times \xi)^*\sM_E\big) \circ \Psi_1:\rd(B)\rightarrow \rd(C)$. 
 As the class of $(1_E\times \xi)^*\sM_E$ is $l(\xi,0,0)$, by Corollary \ref{corPicardprod} we have that 
 $\Psi_2^{\rm CH}$ sends the vector $(0,0,1)$ to 
 $$v_2\; := \; \big(r_1, \, \lambda^*l(\varphi \, +\, r_1 \, \xi,\,  m_1  , \, m_2), \, \chi_2\big)$$ for some integer $\chi_2$. 
 Recall that if $f_1$ and $f_2$ are two isogenies of elliptic curves with $f_1$ inseparable, then $f_1+f_2$ is separable if and only if $f_2$ is 
 separable.
 Then as $r_1$ is relatively prime with $p$, the isogeny $\gamma:=\varphi \, + \,r_2 \, \xi$ is separable 
 (indeed the composition of separable isogenies is separable, 
 and the multiplication-by-$n$ map 
 is separable if and only if $n$ is coprime with $p$), and hence \'{e}tale.
 
 Let $w:=(s, h,\zeta)=\Psi_2^{\rm CH}(1,0,0)$ and note that $\langle v_2, w\rangle_C=1$. 
 By Proposition \ref{Neron} the class $p_2^2 \, h$ (where $p_2={\rm deg}\nu$) 
 is in the image of the pull-back $\lambda^*:{\rm CH}^1(E\times F)\rightarrow {\rm CH}^1(C)$.
 Therefore we can write 
 $$p_2^2 \, h \; = \;\lambda^*l(\psi,\, k_1, \, k_2)$$ for some morphism $\psi$ and integers $k_1,k_2$.  Consider the equivalence 
 $\Psi_3 := T_C(H^{\otimes (n \, p\, p_2^2)})\circ \Psi_2$ which sends the vector $(0,0,1)$ to 
 $$v_3 \; := \; \Big(r_1, \, \lambda^*l \big(\gamma \, +\, r_1 \, p\, n\, \psi, \, m_1 \, +\, r_1 \, p\, n\, k_1, \, m_2\, +\, r_1\, p\, n\, k_2 
 \big), \, \chi_2 \, +\, p\, p_2^2 \, n\, I \, + \, r_1\, p^2 \, p_2^4\, n^2\, \frac{h^2}{2}\Big)$$ 
 where $$I\; := \: \lambda^*(\gamma,\, m_1,\, m_2) \cdot h.$$
 As $(r_1 \, p\, n)\psi$ is either the zero morphism or a non-separable one, 
 we have that $\gamma \, + \, (r_1 \, 
 p \, n\, )\psi$ 
 is a separable isogeny. 
 Choose now a positive integer $n$ so that 
 $$p \, p_2^2 \, n\, I\, +\, \chi_2 \; \not \equiv \; 0 \quad (\mbox{mod }q) \quad \quad \mbox{for every prime divisor } 
 q\neq p_2\mbox{ of }r_2\mbox{ that does not divide }I. $$
 With the same argument of Proposition \ref{RelPrime2}, this choice of $n$ ensures that the third components of $v_3$ 
 is relatively prime with $r_1$.
 
 Now we define the line bundle $\Theta:=\sO_E(O_E)\boxtimes \sO_F(O_F)$ on $E\times F$ and we consider the 
 equivalence $T_C(\lambda^*\Theta^{\otimes (r_1 \, d)})\circ \Psi_3$ with $d\gg 0$ a positive integer. 
 By setting $\theta:=l(0_E,1,1)$ for the class of $\Theta$, 
 the equivalence $T_C(\lambda^*\Theta^{\otimes (r_1 \, d)})\circ \Psi_3$ sends $(0,0,1)$ to 
 $$v_4 \; = \; \Big(r_1, \, \lambda^*(\gamma \, +\, r_1\, n \, p \, \psi,m_1 \, +\, r_1\, p\, n\, k_1\, +\, r_1^2\, d, \, m_2\, +\, 
 r_2 \, p\, n\, k_2 \, +\, r_1^2 \,d), \chi_4\Big)
 $$
 where 
 $$\chi_4 :=  \chi_2 \, +\, p\, p_2^2 \, n\, I \, + \, r_1 \, p^2 \, p_2^4 \, n^2\, \frac{h^2}{2} \, +\, r_1 \, d 
 \big(\theta \cdot \lambda^*l(\gamma \, +\, r_1 \, p\, n \, \psi,  m_1 \, +\, r_1 \, p\, n\, k_1,   m_2 \, +\, r_1\, p\, n\, k_2)\big) \, +\, 
 r_1^3 \, d^2 \, \frac{\theta^2}{2}.$$
 As $\chi_4$ is congruent to the third component of $v_3$ modulo $r_1$, we have that $\chi_4$ is still relatively prime with $r_1$. Moreover, 
 for $d$ large the second component of $v_4$ is ample. 
 Finally $r_1$ is relatively prime with $p$, and hence $T_C(\lambda^*\Theta^{\otimes (r_1 \,d)})\circ \Psi_3$
 is the equivalence we are looking for.
 
 To conclude the proof we need to analyze the case when $\varphi$ is separable. 
 But in this case the proof is simpler as there is no need to introduce 
 the isogeny $\xi$ and the equivalence $\Psi_2$. Then it is enough to set $\gamma=\varphi$ and proceed as in the inseparable case.
\end{proof}

\begin{prop}
Assume the assumptions of Setting \ref{sett} and let $\Phi:\rd(B)\rightarrow \rd(A)$ be an equivalence of derived categories of abelian surfaces.
Moreover assume that the elliptic curves $E$ and $F$ are both supersingular. Then the hypotheses of Proposition \ref{cond&stat} hold.
\end{prop}

\begin{proof}
 By Propositions \ref{RelPrime3} and \ref{chiandr2} we can find an equivalence $\Xi:\rd(B)\rightarrow \rd(C)$ 
 with $C\in \{A,\widehat{A}\}$ such that the vector 
 $$v \; := \; \big(r, \, \lambda^*l(\varphi, m_1,m_2), \, \chi \big) \; = \; \Xi^{\rm CH}(0, \,0, \,1)$$ 
 satisfies the assumptions $(E_1)$, $(E_2)$, and $(E_3)$ of Proposition \ref{cond&stat}.
 The class $l(\varphi,m_1,m_2)$ is the class of an ample line bundle $L(\varphi,M_1,M_2)$ on $E\times F$ where $M_1$ and $M_2$ are line bundles 
 on $E$ and $F$ respectively,
 and $\varphi: F\rightarrow E$ is either the constant morphism $O_E$ in case $\nu$ is an isomorphism, or an \'{e}tale isogeny otherwise. 
Let $\sE\rightarrow W$ be a lift of $E$ to the ring of Witt vectors. 
 If $\varphi=O_E$ is the constant morphism, then on any lift $\sF\rightarrow W$ of $F$ we can construct a lift of $L=M_1\boxtimes M_2$ by 
 simply lifting $M_1$ and $M_2$ to $W$ and taking their exterior product.
 Assume now that $\varphi$ is a separable isogeny. Then Theorem \ref{etalecover} determines a projective lift $\sF\rightarrow W$ of $\sF$ and an \'{e}tale cover
 $\widetilde{\varphi}:\sF\rightarrow \sE$ that specializes to $\varphi$.
 It is not difficult to check that $L(\varphi,M_1,M_2)$ lifts to a line bundle $\mathcal{L}$ on $\sE\times_W\sF$. This goes as follows. 
 As line bundles on curves 
 lift, we can find a lift of $\sO_E(O_E)$ which in turn allows us to construct a lift $\sM_{E,W}$ to $W$ of the 
 Mumford line bundle $\sM_E$ (\emph{cf}. \S \ref{mumlin}). By denoting by $M_{1,W}$ and $M_{2,W}$ the liftings of $M_1$ and $M_2$ on $\sE$ and $\sF$ 
 respectively, we see that the line
 bundle $\mathcal{L}:=(1_{\sE} \times_W \widetilde{\varphi})^*\sM_{E,W} \otimes {\rm pr}_{\sE}^*M_{1,W}\otimes {\rm pr}_{\sF}^*M_{2,W}$ is a lift 
 of $L(\varphi,M_1,M_2)$, where ${\rm pr}_{\sE}$ and ${\rm pr}_{\sF}$ are the obvious projections. In conclusion the line bundle 
 $\lambda_W^*\sL$ lifts $\lambda^*L(\varphi,M_1,M_2)$ where $\lambda_W:\mathcal{C}\rightarrow \sE \times_W\sF$ is the 
 lift of Proposition \ref{cond&stat}. This proves the condition $(A_1)$ in the case ${\rm deg}\,=1,2$. 

 Now suppose ${\rm deg}\,\nu=3$ and that $F$ has an automorphism of groups $\rho:F\rightarrow F$ of order three.
  By a result of Deuring (\emph{cf}. \cite[p. 189 or p. 172]{Oo})
  it is possible to lift a supersingular elliptic curve together with an endomorphism over a finite ramified extension $W$ of the ring of Witt 
  vectors.
 Thus we denote by $\sF_0\rightarrow W$  
 and $\rho_W:\sF_0\rightarrow \sF_0$ the lifts of $F$ and $\rho$ to $W$ respectively. 
 The \'{e}tale cover $\widehat{\varphi}:E\rightarrow F$ determines a lift $\sE_0$ of $E$ together with a lift 
 $\widetilde{\varphi}_0$ of $\widehat{\varphi}$ over $W$.  Denote now by 
 $\lambda_{W,0}:\mathcal{C}_0\rightarrow \sE_0\times_W\sF_0$ the lift of $\lambda$ as defined in Proposition \ref{cond&stat}. 
After noting the isomorphism $(1_E\times \varphi)^*\sM_E\simeq (\widehat{\varphi}\times 1_F)^*\sM_F$, and hence that 
$L(\varphi,M_1,M_2)\simeq (\widehat{\varphi}\times 1_F)^*\sM_F\otimes {\rm pr}_{E}^*M_1\otimes {\rm pr}_F^*M_2$, 
the previous argument shows that the line bundle $\lambda^*L(\varphi,M_1,M_2)$ lifts to $\mathcal{C}_0$.
 Moreover the restriction of $\rho_W$ to the geometric generic fiber $(\sF_0)_{\eta}$ is an automorphism that is not proportional to the identity.
Hence 
  ${\rm rk \,CH}^1\big((\sE_0)_{\eta}\times (\sF_0)_{\eta} \big)=\{2,4\}$ according to whether or not there is an isogeny between 
  $(\sE_0)_{\eta}$ and $(\sF_0)_{\eta}$. As 
 the rank of the N\'{e}ron--Severi group does not change under separable isogenies, 
 we have that ${\rm rk \, CH}^1((\sC_0)_{\eta})\in \{2,4\}$ where $(\sC_0)_{\eta}$ is the geometric generic fiber of $\sC_0$. In particular
 the assumption $(A_2)$ holds.
 \end{proof}
 
\section{Canonical Covers of Enriques Surfaces} \label{FMenr}
 An \emph{Enriques surface} $S$ over an algebraically closed field $k$ of characteristic $p\neq 2$ is a smooth projective minimal 
surface with canonical 
bundle $\omega_S$ of order two, $\chi(\sO_S)=1$, and $b_2(S)=10$. The canonical cover $\pi:X\rightarrow S$ induced by $\omega_S$ is 
a double \'{e}tale cover with $X$ a $K3$ surface.
Let $\Phi:\rd(Y)\rightarrow \rd(X)$ be a derived equivalence so that $Y$ is itself a $K3$ surface by \cite[Proposition 3.9]{LO}.
We aim to prove Theorem \ref{main3}, namely that $X\simeq Y$.

In \cite{LO} the authors prove that Shioda-supersingular $K3$ surfaces do not admit any non-trivial Fourier--Mukai partners, thus we can assume that $X$ has Picard rank less than $22$. Since in odd characteristic Shioda-supersingularity is equivalent to Artin-supersingularity (\emph{cf.} for instance \cite{Pe2015}), we can assume that the formal Brauer group of $X$ is of finite height.  In particular there
exists a lift $\mathcal{X}\rightarrow W$ of $X$ over the ring of Witt vectors 
such that ${\rm NS}(X)\simeq {\rm NS}(\mathcal{X}_\eta)$, where as usual $\mathcal{X}_\eta$ is the geometric generic fiber (\cite[p. 505]{NygO}).  By the work of Jang (\cite[Theorem 2.5]{Ja2013}) there is a primitive embedding 
of the Enriques lattice $\Gamma(2):=U(2)\oplus E_8(-2)$ of $S$ into ${\rm NS}(X)$ such that the orthogonal complement 
of the embedding does not contain any vector of self intersection $-2$. Using the isomorphism ${\rm NS}(X)\simeq {\rm NS}(\mathcal{X}_\eta)$, together with the characterization of Enriques-$K3$ surfaces in terms of its periods (\cite[Theorem 1.14]{Na1985}), we deduce that $\mathcal{X}_\eta$ is a $K3$ surface arising as the canonical cover of an Enriques surface. In particular $\mathcal{X}_\eta$ does not have any non-trivial Fourier--Mukai partner thanks to \cite[Theorem 1.1]{So}. On the other hand,  \cite[Proposition 8.2]{LO} implies that $Y$ is isomorphic to a moduli space $\mathcal{M}_h(v)$ of $h$-Gieseker-stable sheaves  for some Mukai vector $v=(r, l, \chi)$. Since all the line bundles on $X$ deform to line bundles over $\mathcal{X}\rightarrow W$, we can consider the relative moduli space $\sM_{\mathcal{X}/W}(r, \widetilde{l}, \chi)\rightarrow W$ where $\widetilde{l}$ is a lift of $l$. The geometric generic fiber of $\mathcal{M}_{\mathcal{X}/W}(r, \tilde{l}, \chi)\rightarrow W$ is by construction a Fourier--Mukai partner of $\mathcal{X}_\eta$ and therefore, by the previous considerations, it is isomorphic to $\mathcal{X}_\eta$.  We conclude by using \cite[Lemma 6.5]{LO} 
saying that the isomorphism between the geometric generic fibers of relative $K3$ surfaces induces an isomorphism between the closed fibers.

\end{document}